\documentclass[reqno, 10pt]{amsart}
\usepackage{amssymb,amsmath,amsfonts,bm, mathtools,dsfont}
\usepackage{dsfont}
\usepackage{epsfig}
\usepackage{graphicx}
\usepackage{esint}
\usepackage{enumerate}
\usepackage{verbatim}
\usepackage{color}
\usepackage{caption}
\usepackage{geometry}
\usepackage{enumitem}

\usepackage{hyperref,  circledsteps,}



\def\R{{\mathbb R}}
\def\S{{\mathbb S}}
\def\N{{\mathbb N}}

\def\l{{\langle }}
\def\r{{\rangle }}

\def\C{{\mathcal C}}
\def\1{{\mathds{1}}}

\def\2{{\tild{K}_2}}
\def\3{{\tild{K}_3}}
\def\4{{\tild{K}_4}}
\def\5{{\tild{K}_5}}
\def\6{{\tild{K}_6}}
\def\7{{\tild{K}_7}}
\def\8{{\tild{K}_8}}

\date{} 

\DeclareMathOperator\support{supp}

\theoremstyle{plain}
\newtheorem{theorem}{Theorem}[section]
\newtheorem{definition}[theorem]{Definition}
\newtheorem{proposition}[theorem]{Proposition}

\newtheorem{lemma}[theorem]{Lemma}
\newtheorem{corollary}[theorem]{Corollary}
\newtheorem{remark}[theorem]{Remark}
 \newcommand{\tild}{\widetilde}

\numberwithin{theorem}{section}
\numberwithin{equation}{section}
\numberwithin{figure}{section}

\let\oldtocsection=\tocsection
 
\let\oldtocsubsection=\tocsubsection
 
\let\oldtocsubsubsection=\tocsubsubsection
 
\renewcommand{\tocsection}[2]{\hspace{0em}\oldtocsection{#1}{#2}}
\renewcommand{\tocsubsection}[2]{\hspace{1em}\oldtocsubsection{#1}{#2}}
\renewcommand{\tocsubsubsection}[2]{\hspace{2em}\oldtocsubsubsection{#1}{#2}}

\begin{document}
\bibliographystyle{plain}

\parskip=4pt

\vspace*{1cm}
\title[Global  existence and uniqueness of solutions to the Boltzmann hierarchy]
{On the global in time existence and uniqueness of solutions to the Boltzmann hierarchy}

\author[Ioakeim Ampatzoglou]{Ioakeim Ampatzoglou}
\address{Ioakeim Ampatzoglou,  
Baruch College, The City University of New York}
\email{ioakeim.ampatzoglou@baruch.cuny.edu}

\author[Joseph K. Miller ]{Joseph K. Miller}
\address{Joseph K. Miller,  
Department of Mathematics, The University of Texas at Austin.}
\email{jkmiller@utexas.edu}

\author[Nata\v{s}a Pavlovi\'{c}]{Nata\v{s}a Pavlovi\'{c}}
\address{Nata\v{s}a Pavlovi\'{c},  
Department of Mathematics, The University of Texas at Austin.}
\email{natasa@math.utexas.edu}

\author[Maja Taskovi\'{c}]{Maja Taskovi\'{c}}
\address{Maja Taskovi\'{c},  
Department of Mathematics, Emory University}
\email{maja.taskovic@emory.edu}
\begin{abstract} 

In this paper we establish the global in time existence and uniqueness of solutions to the Boltzmann hierarchy, a hierarchy of equations instrumental for the rigorous derivation of the Boltzmann equation from many particles. Inspired by available $L^{\infty}$-based a-priori estimate for solutions to the Boltzmann equation, we develop the polynomially weighted $L^\infty$ a-priori bounds for solutions to the Boltzmann hierarchy and handle the factorial growth of the number of terms in the Dyson's series by reorganizing the sum through a combinatorial technique known as the Klainerman-Machedon board game argument. This paper is the first work that exploits such a combinatorial technique in conjunction with an $L^{\infty}$-based estimate to prove uniqueness of the mild solutions to the Boltzmann hierarchy.
Our proof of existence of global in time mild solutions to the Boltzmann hierarchy for admissible initial data is constructive and it employs known global in time solutions to the Boltzmann equation via a Hewitt-Savage type theorem.

\end{abstract}
\maketitle
\tableofcontents
\section{Introduction}\label{sec-intro}

In this paper, we address the global in time existence and uniqueness of solutions to the Boltzmann hierarchy in space-velocity polynomially weighted $L^\infty$-spaces. The Boltzmann hierarchy  for a sequence of functions $(f^{(k)})_{k=1}^\infty$, with $f^{(k)}:[0,\infty)\times\R^{dk}\times\R^{dk}\to \R$, is a linear coupled system of partial differential equations (PDE)  given by
\begin{align*}
\partial_t f^{(k)}+\sum_{i=1}^k v_i\cdot\nabla_{x_i}f^{(k)}=\mathcal{C}^{k+1}f^{(k+1)}
\end{align*}
(for more details see \eqref{BH}-\eqref{cross-section form}).  The Boltzmann hierarchy has been a central object in the rigorous derivation of the Boltzmann equation \cite{Boltzmann, ma867}
\begin{align*}
  \partial_t f+ v\cdot\nabla_{x}f=Q(f,f)
\end{align*}
from the infinite particle limit of interacting particle systems. For more details about the Boltzmann equation see Section \ref{ssec:Boltzmann equation}. The derivation of the Boltzmann equation was pioneered by Lanford \cite{la75} for systems of hard spheres, and then extended by King \cite{ki75} for short range potential gases.  The program of Lanford  was more recently revisited and refined by Gallagher, Saint-Raymond and Texier \cite{GSRT13} who completed the derivation for hard spheres and short range potentials in full rigor. All these works establish the validity of the equation only for short times. See also \cite{pusasi14} for a derivation for short range potential gases. In addition, see \cite{ampa21,ampa20, am20} for a derivation of a Boltzmann-type of equation involving higher order collisions, and \cite{ammipa22} for a derivation of a Boltzmann system for mixtures of hard sphere gases.

The strategy of Lanford's program consists of identifying a linear finite hierarchy of coupled equations satisfied by the marginals  of a finite system of $N$ particles of radius $\epsilon$, the so called BBGKY hierarchy, and examine its behavior as the number of particles $N\to\infty$ and their radius $\epsilon\to 0^+$.  Then, in the Boltzmann-Grad \cite{Grad 1, Grad 2} scaling $N\epsilon^{d-1}\simeq 1$, one can see that the BBGKY hierarchy formally converges to an infinite linear hierarchy of coupled equations, the so called Boltzmann hierarchy. Under the assumption of propagation of chaos, i.e. the assumption that initially independent  states remain independent under time evolution, the Boltzmann hierarchy reduces to an effective nonlinear equation: the Boltzmann equation. 

The main challenge in the derivation of the Boltzmann equation is the rigorous justification of the convergence of the BBGKY hierarchy to the Boltzmann hierarchy, and consequently up to which timescale this convergence is valid. In both \cite{la75, GSRT13},  convergence is established for short times, comparable only to the first collision time.    One of the fundamental obstructions in extending the convergence to larger times is the factorial growth of the number of terms in the Dyson's 
series expansion of the hierarchies, making even well-posedness very hard to study. It is worth mentioning that Illner and Pulvirenti \cite{ilpu89} were able to reach large times but only for initial data exponentially close to vacuum, an assumption which  simplifies significantly the combinatorics of the Dyson's series expansion.

In contrast, at the level of the Boltzmann equation (in addition to its gas mixtures or higher order correction variants) the long time behavior is much better understood. The case of the Boltzmann equation that is relevant to our work is the cutoff regime with small initial data.   Global existence and uniqueness of non-negative mild solutions in this regime (or close to a Maxwellian) has been addressed in a series of works, a non-exhaustive list being \cite{kaniel-shinbrot, illner-shinbrot, hamdache, beto85, to86,to88, pato88, po88, go97, alonso, alonso-gamba, amgapata22, hanoyu07}. Out of those results,  the ones in $L^\infty$ spaces with polynomial weights allow  slower  decay (compared to exponential weights) on the tails, therefore more collisions happening before the transport forces the solution to disperse. 
Along these lines, Bellomo-Toscani \cite{beto85} showed well-posedness assuming polynomial decay in space and exponential decay in velocity, and later Toscani \cite{to86} extended the results to velocity polynomial decay as well. Polewczak \cite{po88} 
also addressed spatial regularity of the solution in this setting.

Motivated by the fact that the global in time existence and uniqueness of solutions to the Boltzmann equation, at least for small initial data in the cutoff regime, is fairly well understood (unlike our current understanding at the level of the hierarchies), in this paper we investigate the global in time well-posedness of the Boltzmann hierarchy, in space-velocity polynomially weighted $L^\infty$-spaces and for a range of values of the chemical potential (for precise statements of these results, see Theorem \ref{Uniqueness theorem} and Theorem \ref{existence theorem BH}). Here, the chemical potential is a parameter representing the amount of energy the addition of a single particle brings to the system.  Our results apply to a wide range of particle interactions varying from moderately soft up to hard spheres;   the validity of our results for hard spheres is of particular importance since this is the model for which the equation has been rigorously derived \cite{la75,GSRT13}. The space-velocity polynomially weighted $L^\infty$-spaces  we work with are inspired by the works \cite{beto85,to86} at the level of the Boltzmann equation.

\subsection*{Results of the paper in a nutshell.}
The two pillars of this paper are \textit{(i)} The proof of uniqueness of mild solutions to the Boltzmann hierarchy \textit{(ii)} The construction of a global in time solution to the Boltzmann hierarchy for admissible initial data. 

\subsubsection*{ (i) Uniqueness via the board game argument}
The main idea for proving uniqueness of solutions is to expand the solution of the Boltzmann hierarchy into a Dyson's series with respect to the initial data, estimate each term of the series, and then combine these term by term estimates to conclude uniqueness. As previously mentioned,  the factorial growth of the number of terms at the level of the hierarchy causes major difficulties in estimating the terms (or so-called collision histories) in order to obtain uniqueness.  We address the factorial growth in the Dyson's series expansion of the solution by employing a combinatorial reorganization of the series using a technique originating from dispersive equations, known as the Klainerman-Machedon board game argument.

 The Klainerman-Machedon board game argument (based on a reformulation of the above mentioned Dyson series using linear algebra and a relatively easy combinatorics to produce an exponential number of terms, rather than factorial) was introduced in \cite{klma08} and was in the same paper combined with an iterative use of a certain $L^2$-based space-time estimate to provide the uniqueness of certain solutions to the Gross-Pitaevskii (GP) hierarchy. The GP hierarchy is an infinite coupled hierarchy of PDEs having an important role in derivation of the nonlinear Schr\"{o}dinger equation (NLS) from quantum many particle systems, which was achieved in the physically relevant case of cubic NLS in 3D for the first time in breakthrough papers of Erd\"os--Schlein--Yau \cite{erscya06, erscya07}. The class of solutions introduced in \cite{klma08} was consequently studied in the context of derivation of NLS in e.g. \cite{kiscst11,chpa-quintic,chpa-spacetime,chho16}. Later on the board game part of the Klainerman-Machedon method  was used in conjunction with quantum de Finetti theorems to provide an alternative proof of derivation of the cubic NLS in 3D starting with \cite{chhapase15}. 

In the context of kinetic equations, the Klainerman-Machedon board-game method was first used by T.Chen-Denlinger-Pavlovi\'c in \cite{chdepa19} to prove local well-posedness of the Boltzmann hierarchy for cutoff Maxwell molecules. More recently, the board game combinatorial argument was used by X.Chen-Holmer  in \cite{chho23} for the derivation of a quantum Boltzmann equation which incorporates a combination of hard sphere and inverse power law potential, not the full hard sphere model though. All the aforementioned works combine the board game argument with an $L^2$-based space-time estimate, that in turn was inspired from techniques stemming from dispersive PDE.

In this paper we prove uniqueness of a space-velocity polynomially weighted $L^\infty$-based solutions of the Boltzmann hierarchy for a range of cutoff kernels, varying from moderately soft up to, and including, hard spheres. We would like to emphasize that this comes into agreement with the classical $L^\infty$ global well-posedness results at the level of the Boltzmann equation itself \cite{kaniel-shinbrot, illner-shinbrot, hamdache, beto85, to86,to88, po88, go97,alonso,alonso-gamba}.

 To obtain the uniqueness result we find a novel way to employ the Klainerman-Machedon board game algorithm with an iterative use of the polynomially weighted $L^{\infty}$ estimate. To achieve this, we develop  a global in  time $L^\infty$ based a-priori estimate on the collisional term of the hierarchy, which we then iteratively use for any collision history, see Proposition \ref{a-priori 1} and  Corollary \ref{a-priori estimate step n}. These estimates, combined with a Klainerman-Machedon type argument and the bounds on the chemical potential, yield uniqueness of solutions globally in time  (see Theorem \ref{Uniqueness theorem}).

\subsubsection*{(ii) Existence of solutions} Upon proving uniqueness, we focus on constructing global in time solutions to the Boltzmann hierarchy for admissible initial data (for a precise meaning see Definition \ref{admissible data}). Admissible data are  essentially the marginals of a probability density of the particle system. It is important to note that our proof of existence of solutions is constructive, i.e. the solution does not come from say a fixed point argument; therefore we have explicit information about its form and its long time behavior. 
More specifically,
\begin{enumerate}
    \item Thanks to admissibility of the initial data $F_0=(f_0^{(k)})_{k=1}^\infty$ of the Boltzmann hierarchy, we can apply the Hewitt-Savage theorem \cite{hs55} to represent such data as a convex combination of tensorized states with respect to a unique Borel probability measure $\pi$ over the set of probability densities $\mathcal{P}$ i.e.
    \begin{equation*}
     f_0^{(k)} = \int_{\mathcal{P}} h_0^{\otimes k} d\pi(h_0).
    \end{equation*}
    Moreover, we show that the measure $\pi$ is supported on a set of probability densities of explicit space-velocity polynomial decay. 
\item Next, for each such  $h_0$ we  solve the Boltzmann equation to produce a global in time solution $h(t)$ of the same polynomial decay as the initial data. We achieve this step by applying the well-posedness result for the Boltzmann equation stated in Theorem \ref{BE is well posed}. For dimension $d=3$, this result was obtained in \cite{to86}; we extended this result to  arbitrary dimension $d\geq 3$.

\item Equipped with the solution of the Boltzmann equation $h(t)$, we construct the solution of the Boltzmann hierarchy $F=(f^{(k)})_{k=1}^\infty$ as follows:
\begin{equation*}
        f^{(k)}(t) : = \int_{\mathcal{P}} h(t)^{\otimes k} d\pi(h_0),\quad  k\in\N.
    \end{equation*}
\item The constructed solution belongs to the space of solutions where the uniqueness is valid according to Theorem \ref{Uniqueness theorem}. Therefore it is unique.

\end{enumerate}

\subsection*{Future directions} It would be interesting to explore
whether a program that is carried out at the level of the Boltzmann hierarchy in this paper can inspire a treatment of the BBGKY hierarchy, i.e. one can ask whether the Boltzmann equation can be rigorously derived globally in time in polynomially weighted $L^\infty$ spaces. However, due to the presence of collisions of particles in the system this is a subtler problem that requires further investigation.

\subsection*{Organization of the paper} Section 2 provides description of the Boltzmann hierarchy, notation used throughout the paper and  statements of the main results. Section 3 focuses on proving uniqueness of mild solutions to the Boltzmann hierarchy as well as the relevant tools needed for our proof such as an a priori estimates on solutions to the Boltzmann hierarchy and a combinatorial argument inspired by the Klainerman-Machedon board game. Section 4 addresses existence of the global in time mild solution to the Boltzmann hierarchy. Since this construction relies on an existence result for the Boltzmann equation and Hewitt-Savage theorem, those results are reviewed as well. Finally, Appendix contains some general convolution estimates and proofs of technical lemmas which are needed for establishing a priori estimates covered in Section 3.

\subsection*{Acknowledgements.}
I.A. gratefully acknowledges support from the NSF grant No. DMS-2206618 and the Simons Collaboration on Wave Turbulence. J.K.M. gratefully acknowledges support from the Provost’s Graduate Excellence Fellowship at The University of Texas at Austin and from the NSF grants No. DMS-1840314 and DMS-2009549. 
N.P. gratefully acknowledges support from the NSF under grants No. DMS-1840314, DMS-2009549 and DMS-2052789. M.T. gratefully acknowledges support from the NSF grant DMS-2206187.

\section{Notation and main results} \label{Notation and main results section}

In this section, we introduce the Boltzmann hierarchy, the main functional spaces, define precisely the notion of a mild solution to the Boltzmann hierarchy \eqref{BH}, and finally state the main results of this paper. 

 \subsection{The Boltzmann hierarchy}
 Let $d\ge 3$. The Boltzmann hierarchy for a sequence $F=(f^{(k)})_{k=1}^\infty$, $f^{(k)}:[0,\infty)\times\R^{dk}\times\R^{dk}\to \R$ with initial data $F_0=(f^{(k)}_0)_{k=1}^\infty$, $f_0^{(k)}:\R^{dk}\times\R^{dk}\to \R$, is given by
\begin{equation}\label{BH}
\begin{cases}
 \partial_t f^{(k)}+\sum_{i=1}^k v_i\cdot\nabla_{x_i}f^{(k)}=\mathcal{C}^{k+1}f^{(k+1)},\\
 \\
 f^{(k)}(t=0)=f^{(k)}_0,
 \end{cases},\quad k\in\N
\end{equation}
where
for each $k \in \N$ we denote by $\mathcal{C}^{k+1}$ the collisional operator acting on $f^{(k+1)}$. The collisional operator is given by
\begin{align}
    \mathcal{C}^{k+1} &: = \sum_{j=1}^{k} \mathcal{C}_{j,k+1},\label{Ck}\\
\mathcal{C}_{j,k+1}&:=\mathcal{C}_{j,k+1}^+-\mathcal{C}_{j,k+1}^-, \label{gain-loss spliting}
\end{align}
where the gain operators $\mathcal{C}_{j,k+1}^+$ and loss operators $\mathcal{C}_{j,k+1}^-$  are respectively given by
\begin{align}
    \mathcal{C}_{j,k+1}^+f^{(k+1)}(X_{k},V_k) &= \int_{\R^d} \int_{\mathbb{S}^{d-1}} B(\sigma, v_{k+1} - v_j) f^{(k+1)}(X_{k},x_j,V_k^{*j},v_{k+1}^*)  d\sigma dv_{k+1},\label{gain operator}\\
  \mathcal{C}_{j,k+1}^-f^{(k+1)}(X_k,V_k) &= \int_{\R^d} \int_{\mathbb{S}^{d-1}} B(\sigma, v_{k+1} - v_j) f^{(k+1)}(X_k,x_j,V_k,v_{k+1})  d\sigma dv_{k+1},\label{loss operator}  
\end{align}
and we use the notation
\begin{align}
       & X_k :=(x_1,...,x_k),\quad V_k : = (v_1, \dots, v_k),\\
     &V^{*j}_k  = (v_1, \dots, v_j^*, \dots, v_k),\\
    &v_j^*=\frac{v_j+v_{k+1}}{2}+\frac{|v_{k+1}-v_j|}{2}\sigma, \label{*j}\\
    &v_{k+1}^*=\frac{v_j+v_{k+1}}{2}-\frac{|v_{k+1}-v_j|}{2}\sigma. \label{*k+1}
\end{align}
Notice that \eqref{*j}, \eqref{*k+1} implies that the collision between the $j$ and the $k+1$ particles is elastic i.e. momentum and energy is conserved:
\begin{align}
 v_j^*+v_{k+1}^*&=v_j+v_{k+1}, \label{conservation of momentum particle j-k+1}   \\
 |v_j^*|^2+|v_{k+1}^*|^2&=|v_j|^2+|v_{k+1}|^2. \label{conservation of energy j-k+1}
\end{align}
Additionally, the collision preserves the precollisional and postcollisional relative velocity magnitude i.e.
\begin{equation}\label{relative velocity magnitude}
|v_{k+1}^*-v_{j}^*|=|v_{k+1}-v_j|.    
\end{equation}
The factor $B:\S^{d-1}\times\R^{d}\to\R$ in the integrand of \eqref{gain operator}-\eqref{loss operator}, is the differential cross-section which expresses the statistical repartition of particles. It is assumed to be of the form
\begin{equation}\label{cross-section form}
 B(\sigma,u):=|u|^{\gamma}b(\hat{u}\cdot\sigma), \quad \gamma\in (1-d,1],   
\end{equation}
where $\gamma\in (1-d,1]$ represents the type of potential considered,  $\hat{u}:=\frac{u}{|u|}\in\mathbb{S}^{d-1}$ is the unit relative velocity of the particles before the collision, $\sigma\in\mathbb{S}^{d-1}$ represents the scattering direction, and $b:[-1,1]\to\R$ is the angular cross-section, expressing the transition probability between particles.

We assume that:
\begin{itemize}
    \item $b$ is measurable, non-negative and even.
    \item $b\in L^\infty([-1,1])$.
\end{itemize}
 The assumption that the angular cross-section is non-negative corresponds to microreversibility of the collisions i.e. the cross-section does not distinguish precollisional from postcollisional configurations.
 The assumption $b\in L^\infty([-1,1])$ corresponds to Grad's cut-off assumption \cite{Grad 1,Grad 2}, since it implies that the collisional operator can be split into gain and loss terms. 

Of course, the assumption of boundedness of $b$ is stronger than merely the integrability on the sphere that is typically required,  but it will be important for controlling the polynomial velocity weights we will consider. Nevertheless it still includes a wide range of potentials; the range $\gamma\in(1-d,0)$ corresponds to moderately soft potentials, the case $\gamma=0$ corresponds to Maxwell molecules, and the range $\gamma\in (0,1]$ corresponds to hard potentials.
In particular, for $\gamma=1$ and $b(z)=\frac{1}{2}$, one recovers the classical hard-sphere model, for which the Boltzmann equation has been rigorously derived for short times from systems of finitely many particles  \cite{la75, GSRT13}.

Now, introducing the transport operator of $k$-particles $(T_k^{s})_{s\in\R}$, acting on a function $g^{(k)}:[0,\infty)\times\R^{dk}\times\R^{dk}\to\R$ as follows:
\begin{equation}\label{definition of transport}
 T_k^{s}g^{(k)}(t,X_k,V_k):=g^{(k)}(t,X_k-sV_k,V_k),
\end{equation}
Duhamel's formula implies that the Boltzmann hierarchy \eqref{BH} can be formally written in mild form as:
\begin{equation}\label{mild form classic}
f^{(k)}(t)=T_k^{t}f_0^{(k)}+\int_0^t T_k^{t-s}\mathcal{C}^{k+1}f^{(k+1)}(s)\,ds,\quad t\in[0,T],\quad k\in\N,
\end{equation}
or equivalently after applying $T_k^{-t}$ to both sides
\begin{equation}\label{mild form we use}
    T_k^{-t}f^{(k)}(t)=f_0^{(k)}+\int_0^t T_k^{-s}\mathcal{C}^{k+1}f^{(k+1)}(s)\,ds,\quad t\in[0,T],\quad  k\in\N,
\end{equation}
The mild formulation \eqref{mild form we use} will be the one used in this paper.

\subsection{The Functional Spaces}
We next define the appropriate spaces for the transported mild solution, which will be the spaces we will mostly work with.

Consider   $p,q>1$ and $\alpha, \beta >0$. For  $k\in\mathbb{N}$, we define the Banach spaces

\begin{equation}\label{X_k space static}
 X_{ p,q, \alpha, \beta}^k:=\left\{ g^{(k)}:\mathbb{R}^{dk}\times\R^{dk}\to\R,\ \text{measurable and symmetric : }\, \|g^{(k)}\|_{k, p,q, \alpha, \beta}<\infty\right\},   \end{equation}
 where the norm $\|\cdot\|_{k, p,q, \alpha, \beta}$ is given by:
 \begin{equation}
     \|g^{(k)}\|_{k, p,q, \alpha, \beta}:=\sup_{X_k, V_{k}} \l\l \alpha X_k\r\r^p \l\l \beta V_k\r\r^q \left|g^{(k)}(X_k,V_k)\right|,
 \end{equation}
 and for $Y_k=(y_1,\cdots,y_k)\in\R^{dk}$, we denote
 $$\l\l Y_k\r\r:=\prod_{i=1}^k\l y_i\r,\quad \l y_i\r:=\sqrt{1+|y_i|^2},\quad i=1,\cdots,k.$$

 By symmetric, we mean that the particles are indistinguishable i.e.
 \begin{equation}\label{symmetry assumption}
g^{(k)}\circ\pi_k=g^{(k)},\,\,\text{for any permutation $\pi_k$ of the $k$-particles}.  
 \end{equation}
For $k=1$, we will slightly abbreviate notation denoting 
\begin{equation}\label{abbreviation static}
X_{ p,q, \alpha, \beta}:=X^1_{ p,q, \alpha, \beta},\quad\|\cdot\|_{ p,q, \alpha, \beta}:=
\|\cdot\|_{1, p,q,\alpha, \beta}.
\end{equation}

Of particular interest will be the tensorized products of a given function $h:\R^{2d}\to\R$ defined by
$$h^{\otimes k}(X_k,V_k)=\prod_{i=1}^k h(x_i,v_i),\quad k\in\N.$$
\begin{remark}\label{remark on tensorization}
We note that given $k\in\N$, $h^{\otimes k}\in X_{p,q,\alpha, \beta}^k$ if and only if $h\in X_{p,q,\alpha, \beta}$. In particular, there holds
\begin{equation}\label{norm of tensors}
        \|h^{\otimes k}\|_{k,p,q,\alpha, \beta}=\|h\|_{ p,q,\alpha, \beta}^k,\quad\forall k\in\N.
    \end{equation}
    Indeed,  we have
    \begin{align*}
      \|h^{\otimes k}\|_{k,p,q,\alpha, \beta}&=\sup_{X_k,V_k}\l\l \alpha X_k\r\r^p\l\l \beta V_k\r\r^q|h^{\otimes k}(X_k,V_k)|\\
      &=\sup_{(x_1,v_1),\cdots(x_k,v_k)\in\R^{2d}}\prod_{i=1}^{k}\l \alpha x_i\r^p\l \beta v_i\r^q|h(x_i,v_i)|\\
      &=\prod_{i=1}^{k} \sup_{(x_i,v_i)\in\R^{2d}}\l \alpha x_i\r^p\l \beta v_i\r^q|h(x_i,v_i)|
      =\|h\|_{p,q,\alpha, \beta}^k.
    \end{align*}

\end{remark}

\begin{remark}\label{remark on tensorization of the transport} We note that the transport operator tensorizes as well. Namely for given $h:\R^{2d}\to\R$, we have
\begin{equation}\label{tensorization of the transport}
T_k^{s}h^{\otimes k}=(T_1^s h)^{\otimes k},\quad\forall \,s\in\R,\,\,\,\forall\, k\in\N.    
\end{equation}
In particular, by \eqref{norm of tensors}, we have
\begin{equation}\label{tensorization of transport norm}
 \|T_k^{s}h^{\otimes k}\|_{k,p,q,\alpha, \beta}=\|T_1^{s}h\|_{p,q,\alpha, \beta}^k,\quad\forall \,s\in\R,\,\,\,\forall\, k\in\N.    
\end{equation}
\end{remark}

Since we will be working with a hierarchy of equations, given $\mu\in\R$, we define the Banach space
\begin{equation}\label{X_infinity static}
 \mathcal{X}_{p,q,\alpha, \beta,\mu}^\infty:=\left\{G=(g^{(k)})_{k=1}^\infty\in\prod_{k=1}^\infty X_{p,q,\alpha,\beta}^k\,:\,\|G\|_{p,q,\alpha, \beta,\mu}<\infty\right\},   
\end{equation}
with norm 
\begin{equation}\label{full sequence norm}
 \|G\|_{p,q,\alpha,\beta, \mu}=\sup_{k\in\N}e^{\mu k}\|g^{(k)}\|_{k,p,q,\alpha,\beta}.   
\end{equation}

The parameter $\mu\in\R$ is related to the chemical potential of the gas, which quantifies the amount of energy the addition of a single particle brings to the system.

\begin{remark}
Notice that for $\mu<\mu'$, we have $\|G\|_{p,q,\alpha,\beta,\mu}\leq \|G\|_{p,q,\alpha,\beta,\mu'} $, thus $\mathcal{X}^\infty_{p,q,\alpha,\beta,\mu'}\subset \mathcal{X}^\infty_{p,q,\alpha,\beta,\mu} $.    
\end{remark}

\begin{remark}
By \eqref{norm of tensors} we have $H=(h^{\otimes k})_{k=1}^\infty\in \mathcal{X}_{p,q,\alpha,\beta,\mu}^\infty$ if and only if $\|h\|_{p,q,\alpha,\beta}\leq e^{-\mu}$. This in turn implies that if $H=(h^{\otimes})_{k=1}^\infty\in\mathcal{X}_{p,q,\alpha,\beta,\mu}^\infty$, then $\|H\|_{p,q,\alpha,\beta,\mu}\leq 1$. This comes into agreement with a general property for admissible data (see Definition \ref{admissible data}) that we prove in Proposition \ref{hewitt savage} using the Hewitt-Savage theorem \cite{hs55}, and point out in Remark \ref{ball=space}.
\end{remark}

Given a time $T>0$, we denote
 \begin{align}
X_{p,q,\alpha,\beta,T}^k&:=C([0,T],X_{p,q,\alpha,\beta}^k),\label{time space k} \\ \mathcal{X}_{p,q,\alpha,\beta,\mu,T}^\infty&:=C([0,T],\mathcal{X}_{p,q,\alpha,\beta,\mu}^\infty),\label{time space sequence}
 \end{align}
 endowed with the usual supremum norms:
 \begin{align}
     |||g^{(k)}(\cdot)|||_{k,p,q,\alpha,\beta,T}&:=\sup_{t\in[0,T]}\|g^{(k)}(t)\|_{k,p,q,\alpha,\beta},\\ |||G(\cdot)|||_{p,q,\alpha,\beta,\mu,T}&:=\sup_{t\in[0,T]}\|G(t)\|_{p,q,\alpha,\beta,\mu}.
 \end{align}
Again, for $k=1$, we slightly abbreviate notation, denoting
\begin{equation}\label{abbreviation time}
\begin{aligned}
X_{p,q,\alpha,\beta,T}&:=C([0,T],X^1_{p,q,\alpha,\beta}),\quad |||\cdot|||_{p,q,\alpha,\beta,T}:=|||\cdot|||_{1,p,q,\alpha,\beta,T}.
\end{aligned}    
\end{equation}

Now, we are in the position to give a precise definition of mild solutions to the Boltzmann hierarchy \eqref{BH}.
\begin{definition} \label{BH mild definition}
Let $T>0$,  $p,q>1$, $\alpha, \beta>0$, and $\mu\in\R$. A sequence $F = (f^{(k)})_{k=1}^\infty$ of measurable functions $f^{(k)}:[0,T]\times\R^{dk}\times\R^{dk}\to\R$ is called a {\bf mild $\mu$-solution} to  the Boltzmann hierarchy \eqref{BH}  in $[0,T]$,  corresponding to the initial data $F_0\in \mathcal{X}_{p,q,\alpha,\beta,\mu}^\infty$, if 
\begin{equation}\label{transport in space}
\mathcal{T}^{-(\cdot)}F(\cdot):=(T_k^{-(\cdot)}f^{(k)}(\cdot))_{k=1}^\infty\in\mathcal{X}_{p,q,\alpha,\beta,\mu,T}^\infty,
\end{equation}
and
\begin{equation}\label{Boltzmann hierarchy k}
T_k^{-t}f^{(k)}(t)=f_0^{(k)}+\int_0^t T_k^{-s}\mathcal{C}^{k+1}f^{(k+1)}(s)\,ds,\quad\forall\, t\in[0,T],\quad\forall k\in\N.
\end{equation}

\end{definition}

\label{sec - main results}

\begin{remark}
Note that the function spaces used above are weighted $L^\infty$, but the collisional operator involves integration over a codimesion-1 submanifold. However, the operators $\mathcal{C}^{k+1}$ can still be rigorously defined as in the erratum of Chapter 5 of \cite{GSRT13}.
\end{remark}

\subsection{Main results} The paper has two main results. The first result addresses  global in time uniqueness of mild $\mu$-solutions to the Boltzmann hierarchy \eqref{BH} for $\mu$ sufficiently large. The second result  establishes the global in time well-posedness of \eqref{BH} for certain admissible (see Definition \ref{admissible data} below) initial data and $\mu$ sufficiently large.

We first state the  uniqueness result:

\begin{theorem}\label{Uniqueness theorem}
Consider the Boltzmann hierarchy \eqref{BH} with the cross section  \eqref{cross-section form}. 
Let $T>0$, $p>1, q>\max\{d-1+\gamma,d-1\}$, and $\alpha,\beta>0$. Consider $\mu\in\R$ with $e^\mu> 4 C_{p,q,\alpha,\beta}$, 
where 
\begin{align}\label{C_p,q}
C_{p,q,\alpha, \beta}&= \frac{8p }{\alpha(p-1)} U_{q} \, \max\{\beta^q, \beta^{-2q}\}\,\|b\|_{L^\infty},
\end{align}
and $U_q$ is the constant of Lemma \ref{lemma on velocities weight}.
Let $F_0=(f_0^{(s)})\in \mathcal{X}_{p,q,\alpha,\beta,\mu}^\infty$, and assume $F=(f^{(k)})_{k=1}^\infty$ is a mild $\mu$-solution of the Boltzmann hierarchy \eqref{BH}. Then $F$ is unique.
\end{theorem}

Before stating the well-posedness result, we define the notion of admissibility:
\begin{definition}[Admissibility]\label{admissible data}
     Let $G=(g^{(k)})_{k=1}^\infty \in \prod_{k=1}^\infty L^1_{X_k,V_k}$. We say that $G$ is \emph{admissible} if for every $k \in \N$ we have      
\begin{equation}
  g^{(k)}\geq 0,
\end{equation}
\begin{equation}
\int_{\R^{2dk}}g^{(k)}\,dX_k\,dV_k=1,    
\end{equation}
    \begin{equation}
        g^{(k)} = \int_{\R^{2d}} g^{(k+1)} dv_{k+1} dx_{k+1},
    \end{equation}
    \begin{equation}
       g^{(k)}\circ \pi_k=g^{(k)},\,\, \text{for any permutation $\pi_k$ of the $k$-particles}.
    \end{equation}
We denote the set of admissible functions as $\mathcal{A}$.
\end{definition}

\begin{remark}\label{remark on non-triviality of A}
 Let $p,q>d$, $\alpha,\beta>0$ and $\mu'\in\R$. We note that $\mathcal{A}\cap \mathcal{X}_{p,q,\alpha, \beta, \mu'}^\infty\neq \emptyset$ if and only if 
 \begin{equation}\label{extra condition on mu}
e^{\mu'}\leq \alpha^{-d} \beta^{-d} I_pI_q,
\end{equation}
where for any $\ell >d$ we define $I_\ell=\int_{\mathbb{R}^d} \l x\r^{-\ell} dx < \infty$.
Indeed, for any $k\in N$, we have
$$
1= \int_{\R^{2dk}} f_0^{(k)} dX_k d V_k \leq e^{-\mu' k } \int_{\R^{2dk}} \l\l \alpha X_k\r\r^{-p}\l\l \beta V_k\r\r^{-q} dX_k dV_k = (e^{-\mu'}\alpha^{-d} \beta^{-d} I_{p}I_{q})^k.
$$
Since $k$ is arbitrary, \eqref{extra condition on mu} follows. Now if \eqref{extra condition on mu} holds, consider $F_0=(f_0^{(k)})_{k=1}^\infty$ to be
$$
f_0^{(k)} = \frac{\alpha^{dk} \beta^{dk}}{ I_p^k I_q^k} \l\l \alpha X_k\r\r^{-p}\l\l \beta V_k\r\r^{-q}.
$$
Then $F_0$ is clearly admissible and 
$\|F_0\|_{ p,q,\alpha, \beta,\mu'}\leq 1$, due to \eqref{extra condition on mu}.

\end{remark}

We are now in the position to state the global in time well-posedness result of this paper:

\begin{theorem}\label{existence theorem BH}
 Consider the Boltzmann hierarchy \eqref{BH} with the cross section \eqref{cross-section form}. 
    Let $T>0$, $p>1$, $q>\max\{d+\gamma-1,d-1\}$, $\alpha, \beta>0$  and $\mu\in\R$ such that
        $e^{\mu}>8C_{p,q,\alpha,\beta},$
    where $C_{p,q,\alpha,\beta}$ is given by \eqref{C_p,q}.
Consider  admissible initial data $F_0=(f_0^{(k)})_{k=1}^\infty\in \mathcal{A}\cap \mathcal{X}_{p,q,\alpha,\beta,\mu'}^\infty$, where $\mu'=\mu+\ln 2$. Then, there exists a unique mild $\mu$-solution $F=(f^{(k)})_{k=1}^\infty$ of the Boltzmann hierarchy \eqref{BH}. In addition, the solution  satisfies the estimate
\begin{equation}\label{sc 2 stability estimate hierarchy}
|||\mathcal{T}^{-(\cdot)}F(\cdot)|||_{p,q,\alpha,\beta,\mu,T}\leq  1.
\end{equation}
 Moreover, if $\gamma\ge 0$, the following $k$-particle conservation laws hold for any $t\in[0,T]$ and a.e. $X_k\in\R^{dk}$:

\begin{align}
    \text{If $p>d,\,q>d+\gamma$}:\quad      \int_{\R^{dk}} f^{(k)}(t,X_k,V_k)  \,dV_k &= 1,\label{conservation of mass: BH}\\
\text{If $p>d,\,q>d+\gamma+1$}:\quad     \int_{\R^{dk}} V_k f^{(k)}(t,X_k,V_k) \,dV_k &= \int_{\R^{dk}} V_k f^{(k)}_0(X_k,V_k)  \,dV_k, \label{conservation of momentum: BH}\\
\text{If $p>d, q>d+\gamma+2$}:\quad     \int_{\R^{dk}} |V_k|^2f^{(k)}(t,X_k,V_k) \,dV_k &=\int_{\R^{dk}} |V_k|^2f_0^{(k)}(X_k,V_k) \,dV_k. \label{conservation of energy: BH}
\end{align}

In the case that the initial data are tensorized i.e. $F_0=(f_0^{\otimes k})_{k=1}^\infty\in\mathcal{X}^\infty_{p,q,\alpha,\beta,\mu'}$, there holds  the  estimate
\begin{equation}\label{stability estimate hierarchy}
|||\mathcal{T}^{-(\cdot)}F(\cdot)|||_{p,q,\alpha,\beta,\mu,T}\leq \|F_0\|_{p,q,\alpha,\beta,\mu'} .  
\end{equation}
\end{theorem}

\begin{remark}\label{remark on alpha}
If in the statement of Theorem \ref{existence theorem BH} we have $p,q>d$,  Remark \ref{remark on non-triviality of A} imposes the extra condition \eqref{extra condition on mu} so that
 $\mathcal{A}\cap \mathcal{X}_{p,q,\alpha,\beta, \mu'}^\infty\neq \emptyset$. 
When combined with the requirement $e^{\mu} > 8C_{p,q,\alpha, \beta}$ we obtain
 $$
 \frac{8p }{\alpha(p-1)} U_{q} \, \max\{\beta^q, \beta^{-2q}\}\,\|b\|_{L^\infty}   < e^{\mu} < \frac{1}{2} \alpha^{-d} \beta^{-d} I_pI_q.
 $$
 We have a nontrivial range for $\mu$ if $\alpha$  is chosen appropriately small.
\end{remark}

\section{Uniqueness of solutions to the Boltzmann hierarchy}

The goal of this section is to prove Theorem \ref{Uniqueness theorem} on   uniqueness mild $\mu$-solutions to the Boltzmann hierarchy \eqref{BH}.
Due to the linearity of the hierarchy, it suffices to show that if $F_0=0$, the only mild $\mu$-solution is $F=0$. Namely we will prove the following theorem.
\begin{theorem}\label{Proposition with 0 initial data}
 Consider the Boltzmann hierarchy \eqref{BH} with the cross section \eqref{cross-section form}. 
   Let $T>0$, $p>1$, $q>\max\{d+\gamma-1,d-1\}$, and $\alpha,\beta>0$. Consider $\mu\in\R$ with $e^\mu>4C_{p,q,\alpha,\beta}$, where $C_{p,q,\alpha,\beta}$ is given by \eqref{C_p,q}. Then the only mild $\mu$-solution $F=(f^{(k)})_{k=1}^\infty$ of the Boltzmann hierarchy \eqref{BH} with zero initial data, 
 is $F=0$.
\end{theorem}
Given this result, the proof of Theorem \ref{Uniqueness theorem} is straightforward:
\begin{proof}[Proof of Theorem \ref{Uniqueness theorem} using Theorem \ref{Proposition with 0 initial data}]
    Let $F,G$ be mild $\mu$-solutions of \eqref{BH} with initial data $F_0\in \mathcal{X}_{p,q,\alpha,\beta,\mu}^\infty$. Define $H=F-G$. Then,  $\mathcal{T}^{-(\cdot)}H(\cdot)=\mathcal{T}^{-(\cdot)}F(\cdot)-\mathcal{T}^{-(\cdot)}G(\cdot)\in \mathcal{X}_{p,q,\alpha,\beta,\mu}^\infty$, since $\mathcal{T}^{-(\cdot)}F(\cdot),\mathcal{T}^{-(\cdot)}G(\cdot)\in \mathcal{X}_{p,q,\alpha,\beta,\mu}^\infty$.
Finally, by linearity, $H=F-G$ is a mild $\mu$-solution of \eqref{BH}, with zero initial data.
   Thus, by Theorem \ref{Proposition with 0 initial data}, we conclude that $H=0$, so $F=G$.
\end{proof}

Therefore, the rest of this section will be devoted to proving Theorem \ref{Proposition with 0 initial data}. In particular, in subsection \ref{subsec-estimates} we present an  priori estimate for the solution of the Boltzmann hierarchy, and its iterative version. This estimate is inspired by analogous estimates at the level of the Boltzmann equation \cite{beto85,to86}. Then, in subsection \ref{subsec-combinatorics}, inspired by the work of Klainerman and Machedon \cite{klma08} on infinite hierarchy appearing in derivation of the nonlinear Schr\"odinger equation from quantum many particle systems, we adapt a board game method (based on linear algebra and combinatorics) that helps us reorganize iterated Duhamel formula of the type \eqref{iteration1}. Finally, in subsection \ref{estimate+comb} we utilize the a priori estimates and the board game argument to prove Theorem \ref{Proposition with 0 initial data}.

Recall by \eqref{Boltzmann hierarchy k}, that the sequence $F=(f^{(k)})_{k=1}^\infty$ is a mild solution of the Boltzmann Hierarchy \eqref{BH} if for every $k \in \N$, we have 
\begin{equation}\label{non-iterated}
   T^{-t}_k f^{(k)}(t) =  f_0^{(k)} + \int_0^{t} T_k^{-s }\mathcal{C}^{k+1} f^{(k+1)}(s) ds, \quad \forall\,t\in[0,T].
\end{equation}
Iterating this formula with zero initial data $f_0 ^{(k)}=0$ yields
\begin{align}\label{iteration1}
T^{-t}_k f^{(k)}(t)
	 &= \int_0^{t}\int_{0}^{t_{k+1}} \cdots \int_0^{t_{k+n-1}}  dt_{k+n} \cdots dt_{k+2} dt_{k+1} \nonumber \\
     & \qquad T_k^{- t_{k+1}} \mathcal{C}^{k+1} T_{k+1}^{t_{k+1}- t_{k+2}} \mathcal{C}^{k+2} \cdots T^{t_{k+n-1} - t_{k+n}}_{k+n-1} \mathcal{C}^{k+n} f^{(k+n)}(t_{k+n}) .
\end{align}

The main difficulty in controlling the right-hand side of the expression above is that $\mathcal{C}^{k+1}$  is a sum of $k$ terms, see \eqref{Ck}, which results in a factorial number of terms in \eqref{iteration1}. In order to overcome this,  the idea is to use two key ingredients:
\begin{itemize}
\item[(i)] A priori estimate on a time integral of $T_k^{- s} \mathcal{C}^\pm_{j,k+1}$, see Proposition \ref{a-priori estimate step one}, and its iterated analogue, see Corollary \ref{a-priori estimate step n}.
\item[(ii)] A board game, which will allow us to reorganize the integral \eqref{iteration1} into equivalence classes, the number of which is bounded by a power instead of a factorial. Within each equivalence class one can apply the a priori estimate  mentioned above.
\end{itemize}

\subsection{A priori estimates}\label{subsec-estimates}
The first key ingredient for proving uniqueness is an a-priori estimate, and its iterative version, at the level of the Boltzmann hierarchy. As mentioned, this estimate is inspired by an analogous nonlinear estimate in \cite{to86} at the level of the Boltzmann equation, see Subsection \ref{ssec:Boltzmann equation} and Subsection \ref{ssec well posedness of BE} for more details.  We now state our basic a-priori estimate:

\begin{proposition}\label{a-priori estimate step one} Consider the operators \eqref{gain operator}, \eqref{loss operator} with the cross section \eqref{cross-section form}. Let $T>0$, $p>1$, $q>\max\{d+\gamma-1,d-1\}$, and $\alpha,\beta>0$. Then, for all $k\in\N$ and $j\in\{1,...,k\},$ there holds the estimate  
\begin{equation}\label{a-priori estimate polynomial}
\left\| \int_0^T T_k^{-s}C_{j,k+1}^{\pm} g^{(k+1)}(s)\,ds\right\|_{k,p,q,\alpha,\beta} \leq  C_{p,q,\alpha,\beta} |||T_{k+1}^{-(\cdot)} g^{(k+1)}(\cdot)|||_{k+1,p,q,\alpha,\beta,T}, 
\end{equation}
where $C_{p,q,\alpha,\beta}$ is given by \eqref{C_p,q}.
\end{proposition}
For the proof of Proposition \ref{a-priori estimate step one} we rely on the two following auxiliary estimates, whose proof for $d=3$ can be found in \cite{beto85} and \cite{to86} respectively. In the Appendix, we extend these results to arbitrary dimension $d\geq 3$.

\begin{lemma}\label{lemma on positions weight}   
Let $p>1$,  $x\in\R^d$. Consider $\xi,\eta\in\R^d$ with $\xi,\eta\neq 0 $ and $\xi\cdot \eta=0$. Then for any $t\ge 0$ there holds the bound
\begin{equation}\label{estimate on I}
    \int_0^t \l  x+s\xi\r^{-p}\l  x+s\eta\r^{-p}\,ds\leq \frac{4p}{p-1}\,\,\frac{\l  x\r^{-p}}{\min\{|\xi|,|\eta|\}}.
\end{equation} 
\end{lemma}

\begin{lemma}\label{lemma on velocities weight}
Let  $q>\max\{d-1+\gamma,d-1\}$. Then there exists a positive constant $U_{q}$ such that
\begin{equation}\label{Uq def}
    \sup_{v\in\R^d}\int_{\R^d\times\mathbb{S}^{d-1}}\frac{|u|^{\gamma-1}}{\sqrt{1-(\hat{u}\cdot\sigma)^2}}\frac{\l v\r^{q}}{\l v^*\r^{q}\l v_1^*\r^{q}}\,d\sigma\,dv_1\leq U_{q},
\end{equation}
where we denote $u=v_1-v$. 
\end{lemma}

\begin{proof}[Proof of Proposition \ref{a-priori estimate step one}]
Fix $k\in\N$ and $j\in\{1,\cdots, k\}$. We prove estimate \eqref{a-priori estimate polynomial} for the gain operator $C_{j,k}^+$. Fix $X_k,V_k\in\R^d$. Then 
\begin{align}
&  \l\l  \alpha X_k\r\r^{ p} \l\l \beta V_k\r\r^{q}
    \left| \int_0^T T_k^{-s}C_{j,k+1}^{+} g^{(k+1)}(s,X_k,V_k)\,ds\right|\nonumber\\
& = \l\l \alpha X_k\r\r^{p}\l\l \beta V_k\r\r^{q}  
    \left|\int_{0}^T  [\mathcal{C}^+_{j,k+1} g^{(k+1)}](s, X_k+ sV_k, V_k) ds\right| \nonumber \\
& \leq\l\l \alpha X_k\r\r^{p}\l\l \beta V_k\r\r^{q} 
    \int_{0}^T \int_{\R^d}\int_{\mathbb{S}^{d-1}} B(\sigma,v_{k+1} -v_j) 
    \left|g^{(k+1)}(s, X_k + sV_k, x_j + sv_j, V_{k}^{*j} ,v_{k+1}^*)\right| 
    d\sigma dv_{k+1} ds \nonumber\\
& =\l\l \alpha X_k\r\r^{p}\l\l \beta V_k\r\r^{q}  
    \int_{0}^T \int_{\R^d}\int_{\mathbb{S}^{d-1}} d\sigma dv_{k+1} ds \,
    B(\sigma,v_{k+1} -v_j) \nonumber \\
&\hspace{4cm} \left|T_{k+1}^{-s}g^{(k+1)}(s, X_k + s(V_k-V_k^{*j}), x_j + s(v_j-v_{k+1}^*),
    V_{k}^{*j} ,v_{k+1}^*) \right| \label{recall notation}\\
&\leq \l \alpha x_j\r^{p}\|b\|_{L^\infty}|||T_{k+1}^{-(\cdot)}g^{(k+1)}(\cdot)|||_{k+1,p,q,\alpha,\beta,T} 
    \int_{\R^d\times\mathbb{S}^{d-1}} |v_{k+1}-v_j|^\gamma  
    \frac{\l \beta v_j\r^{q}}{\l \beta v_j^*\r^{q}\l \beta v_{k+1}^*\r^{q}}\nonumber \\
&\hspace{1in}\times \left(\int_0^T\l \alpha x_j+ \alpha s(v_j-v_j^*)\r^{-p}\l \alpha x_j+ \alpha s(v_j-v_{k+1}^*)\r^{-p}\,ds\right)
    \,d\sigma\,dv_{k+1}.\label{a-priori 1}
\end{align}
 Since the $j$ and $k+1$ particles are colliding, conservation of momentum and energy yields
\begin{equation}\label{orthogonality condition a-priori}
(v_{j}-v_j^*)\cdot(v_j-v_{k+1}^*)=0.
\end{equation}
Let us write $u_{j,k+1}:=v_{k+1}-v_j$. Then \eqref{orthogonality condition a-priori} and \eqref{relative velocity magnitude} imply
\begin{equation}\label{orthogonality equality a-priori}
|v_j^*-v_j|^2+|v_{k+1}^*-v_j|^2=|v_{k+1}^*-v_{j}^*|^2=|u_{j,k+1}|^2.    
\end{equation}
Moreover, by the collisional formulas we have
\begin{align*}
|v_j-v_j^*|&=\frac{|u_{j,k+1}|}{2}|\hat{u}_{j,k+1}+\sigma|,\quad
|v_j-v_{k+1}^*|=\frac{|u_{j,k+1}|}{2}|\hat{u}_{j,k+1}-\sigma|,
\end{align*}
thus 
\begin{equation}\label{Carleman formula}
|v_{j}^*-v_{j}|\,|v_{k+1}^*-v_j|=\frac{|u_{j,k+1}|^2}{2}\sqrt{1-(\hat{u}_{j, k+1}\cdot\sigma)^2}.
\end{equation}
Therefore using the elementary inequality $\min\{|x|,|y|\}\geq\frac{|xy|}{\sqrt{x^2+y^2}}$ and \eqref{orthogonality equality a-priori},\eqref{Carleman formula}, we obtain
$$\min\{|v_j-v_{j}^*|,|v_j-v_{k+1}^*|\}\geq \frac{|u_{j,k+1}|}{2}\sqrt{1-(\hat{u}_{j,k+1}\cdot\sigma)^2}.$$
Applying Lemma \ref{lemma on positions weight} for $x =\alpha x_j$, $\xi=\alpha(v_j-v_j^*)$, $\eta=\alpha(v_{j}-v_{k+1}^*)$, and using the above estimate, we obtain
\begin{align} \int_0^T\l \alpha x_j + \alpha s(v_j-v_j^*)\r^{-p} 
\l \alpha x_j + \alpha s(v_j-v_{k+1}^*)\r^{-p}\,ds
   &\le  \frac{8p}{\alpha(p-1)}\,\frac{ \l \alpha x_j\r^{-p}}{ |u_{j,k+1}|\sqrt{1-(\hat{u}_{j,k+1}\cdot\sigma)^2}}\label{a-priori 2}.
\end{align}

Combining \eqref{a-priori 1} and \eqref{a-priori 2}, we obtain
\begin{align}
&\l\l  \alpha X_k\r\r^{ p} \l\l \beta V_k\r\r^{ q}\left| 
    \int_0^T T_k^{-s}C_{j,k+1}^{+} g^{(k+1)}(s)\,ds\right|\nonumber\\
&\leq \frac{8p}{\alpha(p-1)}\|b\|_{L^\infty} 
    |||T_{k+1}^{-(\cdot)}g^{(k+1)}(\cdot)|||_{k+1,p,q,\alpha,\beta,T} \int_{\R^{d}\times\S^{d-1}}\frac{|u_{j,k+1}|^{\gamma-1}}{\sqrt{1-(\hat{u}_{j,k+1}\cdot\sigma)^2}}
    \frac{\l \beta v_j\r^{q}}{\l \beta v_j^*\r^{q}\l \beta v_{k+1}^*\r^{q}}\,d\sigma\,dv_1\nonumber\\
&\leq \frac{8p }{\alpha(p-1)} U_{q} \, \max\{\beta^q, \beta^{-2q}\}\,\|b\|_{L^\infty}  |||T_{k+1}^{-(\cdot)}g^{(k+1)}(\cdot)|||_{k+1,p,q,\alpha,\beta,T}\label{use of velocity lemma}\\
&:=C_{ p,q, \alpha, \beta} |||T_{k+1}^{-(\cdot)}g^{(k+1)}(\cdot)|||_{k+1,p,q,\alpha,\beta,T} ,\nonumber
\end{align}
where  to obtain  \eqref{use of velocity lemma} we used Lemma \ref{lemma on velocities weight} and the fact that 
$\frac{\l \beta v_j\r^{q}}{\l \beta v_j^*\r^{q}\l \beta v_{k+1}^*\r^{q}} 
\le \max\{\beta^q, \beta^{-2q}\} \frac{\l v_j\r^{q}}{\l  v_j^*\r^{q}\l v_{k+1}^*\r^{q}}$. Since $X_k,V_k$ were arbitrary, estimate \eqref{a-priori estimate polynomial} follows. 

The estimate for the loss term $C^-_{j,k}$ follows in a similar manner without the need to use Lemma \ref{lemma on positions weight} and Lemma \ref{lemma on velocities weight}

\end{proof}
Recursively applying Proposition \ref{a-priori estimate step one} we obtain the following iterative estimate, which will be useful when bounding the series expansion of the solution.

\begin{corollary}\label{a-priori estimate step n}
Consider the operators \eqref{gain operator}, \eqref{loss operator} with the cross section \eqref{cross-section form}. Let $T>0$, $p,q>d$ and $\alpha,\beta>0$. Then, for all $k,n\in\N$, $\ell\in\{1,...,n\}$, $j_\ell\in\{1,...,k+\ell-1\}$ and $\pi_\ell\in\{+,-\}$, the following estimate holds
\begin{equation}\label{iterated a priori}
\begin{aligned}
\left\|\int_{[0,T]^n} T_{k}^{-t_{k+1}}\mathcal{C}_{j_1,k+1}^{\pi_1}T_{k+1}^{t_{k+1}-t_{k+2}}\mathcal{C}_{j_2,k+2}^{\pi_2}\dots T_{k+n-1}^{t_{k+n-1}-t_{k+n}}\mathcal{C}_{j_n,k+n}^{\pi_n}g^{(k+n)}(t_{k+n})\,dt_{k+n}\dots\,dt_{k+1}  \right\|_{k,p,q,\alpha,\beta} \\
\leq C_{p,q,\alpha,\beta}^n 
\left|\left|\left|T_{k+n}^{-(\cdot)}g^{(k+n)}(\cdot)\right|\right|\right|_{k+n,p,q,\alpha,\beta,T},
\end{aligned}
\end{equation}
where $C_{p,q,\alpha,\beta}$ is given by \eqref{C_p,q}.
\end{corollary}

\begin{proof}
Let $k,n\in\N$, $\ell\in\{1,...,n\}$, $j_\ell\in\{1,...,k+\ell-1\}$ and $\pi_\ell\in\{+,-\}$. Applying  Proposition \ref{a-priori estimate step one} iteratively, we get
\begin{align}
&\left\|\int_{[0,T]^n} T_{k}^{-t_{k+1}}\mathcal{C}_{j_1,k+1}^{\pi_1}T_{k+1}^{t_{k+1}-t_{k+2}}\mathcal{C}_{j_2,k+2}^{\pi_2}\dots T_{k+n-1}^{t_{k+n-1}-t_{k+n}}\mathcal{C}_{j_n,k+n}^{\pi_n}g^{(k+n)}(t_{k+n})\,dt_{k+n}\dots\,dt_{k+1} \right\|_{k,p,q,\alpha,\beta} \nonumber \\
&= \left\|\int_0^T  T_{k}^{-t_{k+1}}\mathcal{C}_{j_1,k+1}^{\pi_1}
\int_{[0,T]^{n-1}}T_{k+1}^{t_{k+1}-t_{k+2}}\mathcal{C}_{j_2,k+2}^{\pi_2}\dots T_{k+n-1}^{t_{k+n-1}-t_{k+n}}\mathcal{C}_{j_n,k+n}^{\pi_n}g^{(k+n)}(t_{k+n})\,dt_{k+n}\dots\, dt_{k+1} \right\|_{k,p,q,\alpha,\beta} \nonumber\\
&:=\left\|\int_0^T T_{k}^{-t_{k+1}}\mathcal{C}_{j_1,k+1}^{\pi_1}G^{(k+1)}(t_{k+1})\,dt_{k+1}\right\|_{k,p,q,\alpha,\beta} \nonumber \\
&\leq C_{p,q,\alpha,\beta} \left|\left|\left|T_{k+1}^{-(\cdot)}G^{(k+1)}(\cdot)\right|\right|\right|_{k+1,p,q,\alpha,\beta,T}\nonumber \\
&=C_{p,q,\alpha,\beta} \sup_{t_{k+1} \in [0,T]}\Bigg\| T_{k+1}^{-t_{k+1}}\int_{[0,T]^{n-1}}T_{k+1}^{t_{k+1}-t_{k+2}}\mathcal{C}_{j_2,k+2}^{\pi_2}\dots\nonumber\\
&\hspace{5cm}\dots T_{k+n-1}^{t_{k+n-1}-t_{k+n}}\mathcal{C}_{j_n,k+n}^{\pi_n}g^{(k+n)}(t_{k+n})\,dt_{k+n}\dots\, dt_{k+2}\Bigg\|_{k+1,p,q,\alpha,\beta} \nonumber\\
&=C_{p,q,\alpha,\beta}\left\|\int_{[0,T]^{n-1}}T_{k+1}^{-t_{k+2}}\mathcal{C}_{j_2,k+2}^{\pi_2}\dots T_{k+n-1}^{t_{k+n-1}-t_{k+n}}\mathcal{C}_{j_n,k+n}^{\pi_n}g^{(k+n)}(t_{k+n})\,dt_{k+n}\dots\, dt_{k+2} \right\|_{k+1,p,q,\alpha,\beta}   \nonumber \\
& \leq \dots \nonumber \\
&=C_{p,q,\alpha,\beta}^2\left\|\int_{[0,T]^{n-2}}T_{k+2}^{-t_{k+3}}\mathcal{C}_{j_3,k+3}^{\pi_3}\dots T_{k+n-1}^{t_{k+n-1}-t_{k+n}}\mathcal{C}_{j_n,k+n}^{\pi_n}g^{(k+n)}(t_{k+n})\,dt_{k+n}\dots\, dt_{k+3} \right\|_{k+2,p,q,\alpha,\beta}   \nonumber \\
& \leq \dots \nonumber \\
&\leq C_{p,q,\alpha,\beta}^n
\left|\left|\left|T_{k+n}^{-(\cdot)}g^{(k+n)}(\cdot)\right|\right|\right|_{k+n,p,q,\alpha,\beta,T}\nonumber.
\end{align}

\end{proof}

\subsection{Reorganization of the integral via the board game}\label{subsec-combinatorics}

In this section we introduce a board game inspired by \cite{klma08}, which will enable us to reorganize the integral \eqref{iteration1} in order to manage the factorial number of terms in the Dyson's series expansion. Motivated by \cite{klma08}, we introduce the following notation for the integrand in \eqref{iteration1}.

\begin{definition}
Let $k,n \in \N$, $p,q>1$, $\alpha,\beta>0$ and  $\underline{t}_{n,k} = (t_{k+1}, \dots, t_{k+n}) \in [0,T]^{n}$. Define the operator  by the expression
\begin{equation}\label{J}
    J_{n,k}(\underline{t}_{n,k})f^{(k+n)} : = 
    T_k^{- t_{k+1}} \mathcal{C}^{k+1} T_{k+1}^{t_{k+1}- t_{k+2}} \mathcal{C}^{k+2} \cdots T^{t_{k+n-1} - t_{k+n}}_{k+n-1} \mathcal{C}^{k+n} f^{(k+n)}(t_{k+n}).
\end{equation}
\end{definition}
\begin{remark}
\begin{itemize}
\item[(i)] The operator $J_{n,k}(\underline{t}_{n,k})$  maps a $X^{k+n}_{p,q,\alpha,\beta,T}$ function to a function that is in  $X^k_{p,q,\alpha,\beta}$ for a.e. $\underline{t}_{n,k}$ thanks to the a priori estimate \eqref{iterated a priori}. 
\item[(ii)] Since $\mathcal{C}^{k+1}  = \sum_{j=1}^{k} \C_{j, k+1}$, we can write 
\begin{equation}\label{sum of Js}
    J_{n,k}(\underline{t}_{n,k}) = \sum_{\mu \in M_n} J_{n,k}(\underline{t}_{n,k}; \mu),
\end{equation}
where the sum is taken over the set of maps
\begin{equation}\label{M_n}
    M_n : = \left\{ \mu : \{k+1,\dots, k+n\} \rightarrow \{1,\dots, k+n-1\} \text{ such that } \,\, \mu(j) < j 
     \text{ for all } j\right\},
\end{equation}
and where
\begin{align}
  & J_{n,k}(\underline{t}_{n,k};\mu)f^{(k+n)} \nonumber\\
   &:= T_k^{- t_{k+1}} \mathcal{C}_{\mu(k+1),k+1} T_{k+1}^{t_{k+1}- t_{k+2}} \mathcal{C}_{\mu(k+2),k+2}  \cdots T^{t_{k+n-1} - t_{k+n}}_{k+n-1} \mathcal{C}_{\mu(k+n), k+n} f^{(k+n)}(t_{k+n}).
   \label{J mu}
\end{align}
\end{itemize}
\end{remark}

Next, we consider time integrals of the $J_{n,k}(\underline{t}_{n,k}; \mu)$ operators. These will be the quantities that are invariant under acceptable board game moves, as will be explained below.

\begin{definition}
\label{I integrals}
Fix $k, n \in \N$ and $\mu \in M_n$, where $M_n$ is defined in \eqref{M_n}. For each $\sigma \in S(\{k+1, \dots,k+ n\})$
define the operator  $\mathcal{I}_{n,k}(\mu,\sigma)$  by the expression 
\begin{equation}
    \mathcal{I}_{n,k}(\mu,\sigma) : = \int_{t \geq t_{\sigma(k+1)} \geq t_{\sigma(k+2)} \geq \dots \geq t_{\sigma(k+n)}\geq 0} J_{n,k}( \underline{t}_{n,k};\mu) dt_{k+n} dt_{k+n-1} \dots dt_{k+1}.
\end{equation}
\end{definition}
Note that it is equivalent to write 
\begin{equation}
    \mathcal{I}_{n,k}(\mu,\sigma)  = \int_{t\geq t_{k+1} \geq t_{k+2} \geq \dots \geq t_{k+n}\geq 0} J_{n,k}( \sigma^{-1}(\underline{t}_{n,k});\mu) dt_{k+n} dt_{k+n-1} \dots dt_{k+1},
\end{equation}
where  
\begin{equation}
    \sigma^{-1}(\underline{t}_{n,k}) : = ( t_{\sigma^{-1}(k+1)}, \dots , t_{\sigma^{-1}(k+n)}).
\end{equation}

\begin{remark}
The operator  $\mathcal{I}_{n,k}(\mu,\sigma)$ maps a $X^{k+n}_{p,q,\alpha,\beta,T}$ function to a function that is in  $X^k_{p,q,\alpha,\beta, T}$ thanks to the a priori estimate \eqref{iterated a priori}.

\end{remark}

The integral $\mathcal{I}_{n,k}(\mu,\sigma)$ is determined by $\mu$ and $\sigma$, and it can be visualized as a $(k+n-1) \times n$ matrix (see \eqref{matrixC}, whose columns are labelled $k+1$ to $k+n$ and whose rows are denoted $1$ to $k+n-1$. In each column  $k+j$, exactly one element is circled that corresponds to the operator  $C_{\mu(k+j), k+j}$  appearing in $J_{n,k}( \underline{t}_{n,k};\mu) $.

\begin{equation} \label{matrixC}
\left[ 
\begin{array}{ccccccc}  
   t_{\sigma^{-1}(k+1)}     &  t_{\sigma^{-1}(k+2)} & ...&  t_{\sigma^{-1}(k+j)} & ...&  t_{\sigma^{-1}(k+n)} & \vspace{8pt} \\
    \Circled{C_{1,k+1} }         & C_{1,k+2}          & ...& ... & ...& \Circled{C_{1, k+n}}  & \text{row } 1\\
    ...                                & \Circled{C_{2,k+2}}  & ...& ... & ...& ... & ... \\
    ...                                & ...                       & ...&  \Circled{C_{\mu(k+j),k+j}} & ...& ... &  \text{row } 2\\
    C_{k, k+1}                   & C_{k,k+2}          & ...& ... & ...& ... &... \\
    0                                 & C_{k+1,k+2}       & ...& ... & ...& ...& ... \\
    ...                                & 0                        & ...& ... & ...& ... & ...\\ 
    ...                                & ...                       & ...& ... & ...& ... & ...\\
    ...                                & ...                       & ...& 0  & ...& ...& ...\\
    0                                 & 0                        & ...& 0 & ...& C_{k+n-1,k+n} & \text{row } k+n-1\vspace{8pt}\\
    \text{col } k+1	     &   \text{col } k+2  & ...&  \text{col } k+j  & ...&  \text{col } k+n &
\end{array} 
\right] \,.
\end{equation}

\noindent Inspired by \cite{klma08}, we define a board game on such matrices with  a set of ``acceptable moves". Imagine a board with carved in names $C_{i,j}$ arranged as in \eqref{matrixC}. We also add a top row to keep track of times as in  \eqref{matrixC}. We associate each $(\mu, \sigma) \in G: = M_n \times S(\{k+1, \dots, k+n\})$ with a "state" of a game. The mapping $\mu$ determines which elements on the board are circled (recall that in each column exactly one element is circled), while the mapping $\sigma$ determines the order of times in the top row. For certain states, we  define "acceptable moves" of the game.  During the game only times and circles can  move positions. Names  $C_{i,j}$ are carved in and do not move.

 If $\mu(j+1) < \mu(j)$, then an ``acceptable move" consists of the following set of operations:
\begin{itemize}
\item[(1)] exchange positions of times in column $j$ and column $j+1$, and
\item[(2)] exchange positions of  circles in column $j$ and  column $j+1$, and
\item[(3)]  exchange  positions of  circles in row $j $ and row $j+1$ if such rows exist. If one of those rows does not exist, no changes are made at the level of rows.
\end{itemize}

Before providing rigorous definition of an acceptable move, let us demonstrate it on two examples. Let $k=2, n=4$ and consider the following state:

\begin{equation}
\left[ 
\begin{array}{ccccc}  
 t_5				& t_3				& t_4 			& t_6   \vspace{4pt} \\
 C_{1,3}			& C_{1,4}			& \Circled{C_{1,5}}	& C_{1,6} & \text{row } 1\\
\Circled{C_{2,3}}	& C_{2,4}			& C_{2,5}			& C_{2,6} & \text{row } 2\\
 0				& \Circled{C_{3,4}}	& C_{3,5}			& C_{3,6} & \text{row } 3\\
 0				& 0				& C_{4,5}			& \Circled{C_{4,6}} & \text{row } 4\\
 0				& 0				& 0				& C_{5,6} & \text{row } 5 \\
 \text{col } 3 & \text{col } 4 & \text{col } 5 & \text{col } 6
 \end{array} 
\right] \,.
\end{equation}
In this example, $1= \mu(5) < \mu(4) =3$. Thus, we exchange positions of times and circles in column 4 and column 5, and we also exchange positions of circles in row 4 and row 5. This results in the new state of the board game: 
\begin{equation}
\left[ 
\begin{array}{cccc}  
 t_5				& t_4				& t_3 			& t_6 \vspace{4pt} \\
 C_{1,3}			& \Circled{C_{1,4}}	&C_{1,5}			& C_{1,6} \\
\Circled{C_{2,3}}	& C_{2,4}			& C_{2,5}			& C_{2,6} \\
 0				& C_{3,4}			& \Circled{C_{3,5}}	& C_{3,6} \\
 0				& 0				& C_{4,5}			& C_{4,6} \\
 0				& 0				& 0				&  \Circled{C_{5,6}}
 \end{array} 
\right] \,.
\end{equation}

Here is one more example, with the same $k=2$ and $n=4$, but where there will be no row exchanges. Namely, consider the state
\begin{equation}
\left[ 
\begin{array}{ccccc}  
 t_5				& t_3				& t_4 			& t_6 \vspace{4pt} \\
 C_{1,3}			& C_{1,4}			& C_{1,5}	& C_{1,6} & \text{row } 1\\
\Circled{C_{2,3}}	& C_{2,4}			& C_{2,5}			& C_{2,6} & \text{row } 2 \\
 0				& \Circled{C_{3,4}}	& C_{3,5}			& \Circled{C_{3,6}} & \text{row } 3 \\
 0				& 0				& \Circled{C_{4,5}}			& C_{4,6} & \text{row } 4 \\
 0				& 0				& 0				& C_{5,6} & \text{row } 5 \\
  \text{col } 3 & \text{col } 4 & \text{col } 5 & \text{col } 6
 \end{array} 
\right] \,.
\end{equation}
Now, times and circles in columns 5 and 6 are being exchanged. But since there is no row 6, no action on rows is taken. So the resulting matrix after the acceptable move is 
\begin{equation}
\left[ 
\begin{array}{cccc}  
 t_5				& t_3				& t_6 			& t_4 \vspace{4pt} \\
 C_{1,3}			& C_{1,4}			& C_{1,5}	& C_{1,6} \\
\Circled{C_{2,3}}	& C_{2,4}			& C_{2,5}			& C_{2,6} \\
 0				& \Circled{C_{3,4}}	& \Circled{C_{3,5}}			& C_{3,6} \\
 0				& 0				& C_{4,5} 		& \Circled{C_{4,6}} \\
 0				& 0				& 0				& C_{5,6}
 \end{array} 
\right] \,.
\end{equation}

\bigskip

We now provide a precise definition of an   "acceptable moves" of the game.

\begin{definition}[Acceptable Moves]
Let $k, n \in \N$ with $n \ge 2$, and recall the definition of $M_n$ in \eqref{M_n}. Let $(\mu, \sigma) \in G  = M_n \times S(\{k+1, \dots, k+n\})$  be a state of the game such that $\mu(j+1) < \mu (j)$ for some $j \in \{k+1, \dots, k+n-1 \}$.  An acceptable move changes $(\mu,\sigma)$ to $(\mu',\sigma')$ by the rule 
\begin{equation}\label{Acceptable Move}
    \mu' = (j, j+1) \circ \mu \circ (j , j+1), \quad \text{and} \quad \sigma' = (j, j+1) \circ \sigma,
\end{equation}
where $(j,j+1)$ is the standard cycle notation for the symmetric group. 
\end{definition}

We next show that the integrals defined in Definition \ref{I integrals} are invariant under acceptable moves.

\begin{proposition}[Acceptable Move Invariance]
 Let $k, n \in \N$ with $n \ge 2$, and  recall the definition of $M_n$ in \eqref{M_n}. Let $(\mu, \sigma) \in G  = M_n \times S(\{k+1, \dots, k+n\})$  be a state of the game such that $\mu(j+1) < \mu (j)$ for some $j \in \{k+1, \dots, k+n-1 \}$.  Define  $(\mu',\sigma')$ by 
\begin{equation}
    \mu' = (j, j+1) \circ \mu \circ (j , j+1), \quad \text{and} \quad \sigma' = (j, j+1) \circ \sigma.
\end{equation}
Then $(\mu', \sigma') \in G$ and
\begin{equation}\label{invariance}
    \mathcal{I}_{n,k}(\mu,\sigma) = \mathcal{I}_{n,k}(\mu',\sigma').
\end{equation}
\end{proposition}

\begin{proof}
 It is easy to verify that $(\mu',\sigma') \in G$, so it remains to prove \eqref{invariance}.
First we write $\mathcal{I}_{n,k}(\mu',\sigma')$, acting on a function $f^{(n+k)} \in X^{k+n}_{p,q,\alpha,\beta,T}$, explicitly: 
\begin{align}
    \mathcal{I} &(\mu',\sigma')f^{(n+k)}  = 
     \int_{t_k \geq t_{\sigma'(k+1)} \geq t_{\sigma'(k+2)} \geq \dots \geq t_{\sigma'(k+n)}\geq 0} J_{n,k}( \underline{t}_{n,k};\mu') f^{(n+k)}\, dt_{k+n}  \dots dt_{k+1}  \nonumber \\
   & = \int_{t_k \geq t_{k+1} \geq t_{k+2} \geq \dots \geq t_{k+n}\geq 0} J_{n,k}( \sigma'^{-1}(\underline{t}_{n,k});\mu') f^{(n+k)} \, dt_{k+n}  \dots dt_{k+1} \nonumber \\
  & = \int_{t_k \geq t_{k+1} \geq t_{k+2} \geq \dots \geq t_{k+n}\geq 0} 
  		 T_k^{- t_{\sigma'^{-1}(k+1)}} \mathcal{C}_{\mu'(k+1),k+1} \, T_{k+1}^{t_{\sigma'^{-1}(k+1)}- t_{\sigma'^{-1}(k+2)}} \mathcal{C}_{\mu'(k+2),k+2} \nonumber  \\
     & \qquad \qquad \dots T_{j-1}^{t_{\sigma'^{-1}(j-1)} - t_{\sigma'^{-1}(j)}}
    \mathcal{C}_{ \mu'(j) ,j } 
    \,\, T_{j}^{t_{\sigma'^{-1}(j)}- t_{\sigma'^{-1}(j+1)}} 
    \mathcal{C}_{ \mu'(j+1),j+1}  \nonumber 
    \,\, T_{j+1}^{t_{\sigma'^{-1}(j+1)} -t_{\sigma'^{-1}(j+2)}} 
  \dots \nonumber \\
  & \qquad \qquad \cdots\,  T^{t_{\sigma'^{-1}(k+n-1)} - t_{\sigma'^{-1}(k+n)}}_{k+n-1} \mathcal{C}_{\mu'(k+n), k+n}
  \, f^{(k+n)}(t_{\sigma'^{-1}(k+n)}) dt_{k+n}  \dots dt_{k+1}.
\end{align}
Since $\mu(j+1) < \mu(j) < j$, by the definition of $\mu'$ in \eqref{Acceptable Move}, we have that $\mu'(j) = \mu(j+1)$. Similarly, since $\mu(j)<j$, by \eqref{Acceptable Move} we have  $\mu'(j+1) =\mu(j)$. For any other $\ell \not \in \{j, j+1\}$, we have $\mu'(\ell) = (j,j+1) \circ \mu(\ell)$. In particular, for $\ell < j$, we have $\mu'(\ell) =  \mu(\ell)$ since $\mu(\ell) < \ell <j$.

\noindent On the other hand,  by \eqref{Acceptable Move}, we have $\sigma'^{-1} = \sigma^{-1} \circ (j, j+1)$, and thus $\sigma'^{-1} (j) = \sigma^{-1} (j+1)$, 
 $\sigma'^{-1} (j+1) = \sigma^{-1} (j)$ and  $\sigma'^{-1} (\ell) = \sigma^{-1} (\ell)$ for $\ell \ne j, j+1$.
Therefore, 
\begin{align}
 \mathcal{I} (\mu',\sigma') f^{(n+k)} &= 
  \int_{t_k \geq t_{k+1} \geq t_{k+2} \geq \dots \geq t_{k+n}\geq 0}
    \,\, T_k^{- t_{\sigma^{-1}(k+1)}} \mathcal{C}_{ \mu(k+1),k+1} 
    \, T_{k+1}^{t_{\sigma^{-1}(k+1)}- t_{\sigma^{-1}(k+2)}} \dots \nonumber \\
  & \quad \dots T_{j-1}^{t_{\sigma^{-1}(j-1)} - t_{\sigma^{-1}(j+1)}}
    \mathcal{C}_{ \mu(j+1) ,j } 
    \,\, T_{j}^{t_{\sigma^{-1}(j+1)}- t_{\sigma^{-1}(j)}} 
    \mathcal{C}_{ \mu(j),j+1}   \,\, T_{j+1}^{t_{\sigma^{-1}(j)} -t_{\sigma^{-1}(j+2)}} 
  \dots \nonumber \\
  & \quad\dots  T^{t_{\sigma^{-1}(k+n-1)} - t_{\sigma^{-1}(k+n)}}_{k+n-1} \mathcal{C}_{(j,j+1) \circ\mu(k+n), k+n}
  \, f^{(k+n)}(t_{\sigma^{-1}(k+n)}) 
  \, dt_{k+n}  \dots dt_{k+1}. \label{I'}
\end{align}

We recall that we need to show
$\mathcal{I}_{n,k}(\mu,\sigma) = \mathcal{I}_{n,k}(\mu',\sigma')$. So, we now expand $ \mathcal{I}_{n,k}(\mu,\sigma)$:
\begin{align}
    \mathcal{I} &(\mu,\sigma) f^{(n+k)} 
   = \int_{t_k \geq t_{k+1} \geq t_{k+2} \geq \dots \geq t_{k+n}\geq 0} 
  		 T_k^{- t_{\sigma^{-1}(k+1)}} \mathcal{C}_{\mu(k+1),k+1} \, T_{k+1}^{t_{\sigma^{-1}(k+1)}- t_{\sigma^{-1}(k+2)}} \dots \nonumber  \\
     & \qquad \qquad \dots T_{j-1}^{t_{\sigma^{-1}(j-1)} - t_{\sigma^{-1}(j)}}
    \mathcal{C}_{ \mu(j) ,j } 
    \,\, T_{j}^{t_{\sigma^{-1}(j)}- t_{\sigma^{-1}(j+1)}} 
    \mathcal{C}_{ \mu(j+1),j+1}  \nonumber 
    \,\, T_{j+1}^{t_{\sigma^{-1}(j+1)} -t_{\sigma^{-1}(j+2)}} 
  \dots \nonumber \\
  & \qquad \qquad \cdots\,  T^{t_{\sigma^{-1}(k+n-1)} - t_{\sigma^{-1}(k+n)}}_{k+n-1} \mathcal{C}_{\mu(k+n), k+n}
  \, f^{(k+n)}(t_{\sigma^{-1}(k+n)}) dt_{k+n}  \dots dt_{k+1}. \label{I}
\end{align}
Note that the terms appearing before $T_{j-1}$ in \eqref{I'} and \eqref{I} match. The terms appearing after the operator $T_{j+1}$ differ only in the first index of $\C$ operators ( $(j,j+1) \circ \mu(\cdot)$ vs. $\mu(\cdot)$). And finally, from $T_{j-1}$ to $T_{j+1}$ there are differences in the indices of $t$'s as well as the indices of $\C$'s.

In order to compare the operators appearing after $T_{j+1}$, we will need the following  lemma.

\begin{lemma}\label{SC lemma}
Let $S_{j, j+1}$ be an operator that exchanges $x$ variables in positions $j$ and $j+1$, and exchanges $v$ variables in positions $j$ and $j+1$. In other words, for $\ell > j$, we define
\begin{align}\label{S operator}
    \Big[S_{j,j+1} f^{(\ell)}\Big](X_\ell, V_\ell):= f^{(\ell)}(x_1, \dots, x_{j+1}, x_j, \dots, x_\ell; \,\, v_1, \dots, v_{j+1}, v_j, \dots, v_\ell).
\end{align}
Then for $\ell > j+1$ we have
\begin{align}\label{claim}
  S_{j, j+1} \mathcal{C}_{(j,j+1) \circ\mu(\ell), \ell}
   =  \mathcal{C}_{\mu(\ell), \ell}\, S_{j, j+1}.
\end{align}
\end{lemma}

\begin{proof}[Proof of Lemma \ref{SC lemma}] To prove this lemma, we consider two cases.

Case 1: assume that $\ell > j+1$ and  $\mu(\ell) \not \in \{j, j+1\}$.  Then  for any $f^{(\ell)} \in X^{\ell}_{p,q,\alpha,\beta,T}$ , we have
\begin{align*}
  \big[S_{j, j+1}& \mathcal{C}_{(j,j+1) \circ\mu(\ell), \ell} f^{(\ell)}\big] (X_{\ell-1}, V_{\ell-1})
  = \big[S_{j, j+1} \mathcal{C}_{\mu(\ell), \ell}f^{(\ell)}\big] (X_{\ell-1}, V_{\ell-1}) \\
 & = \big[ \mathcal{C}_{\mu(\ell), \ell}f^{(\ell)}\big] (x_1, \dots, x_{j+1},  x_{j}, \dots, x_{\ell-1}; \,\, v_1, \dots, v_{j+1}, v_{j}, \dots, v_{\ell-1} )\\
   & = \int_{\R^d} \int_{\S^{d-1}} B(\sigma, v_\ell - v_{\mu(\ell)}) \\
   &  \qquad  \Big(
   f^{(\ell)} (x_1, \dots, x_{j+1}, x_{j}, \dots, x_\ell; \,\, v_1, \dots, v_{j+1}, v_j, \dots, v_{\mu(\ell)}^*,  \dots, v_\ell^* )\\
   &\qquad 
    - f^{(\ell)} (x_1, \dots, x_{j+1},  x_{j}, \dots, x_\ell; \,\, v_1, \dots, v_{j+1}, v_j, \dots, v_{\mu(\ell)},  \dots, v_\ell )
   \Big) \, d\sigma dv_\ell\\
& =  \int_{\R^d} \int_{\S^{d-1}} B(\sigma, v_\ell - v_{\mu(\ell)}) \\
& \qquad \Big( [S_{j, j+1}f^{(\ell)}](x_1, \dots, x_{j},  x_{j+1}, \dots, x_\ell; \,\, v_1, \dots, v_{j}, v_{j+1}, \dots, v_{\mu(\ell)}^*, \dots,  \dots, v_\ell^* )\\
&\qquad 
    -[S_{j, j+1}f^{(\ell)}](x_1, \dots, x_{j},  x_{j+1}, \dots, x_\ell; \,\, v_1, \dots, v_{j}, v_{j+1}, \dots, v_{\mu(\ell)},  \dots, v_\ell )
   \Big) \, d\sigma dv_\ell\\
& = \big[ \mathcal{C}_{\mu(\ell), \ell} S_{j, j+1}f^{(\ell)}\big](X_{\ell-1}, V_{\ell-1}),
\end{align*}
which proves the claim in case 1.

Case 2: assume that $\ell > j+1$ and $\mu(\ell)  \in \{j, j+1\}$. For example, assume that $\mu(\ell) = j$. The other case ($\mu(\ell) = j+1$) is treated analogously. Then  for any $f^{(\ell)}\in X^{\ell}_{p,q,\alpha,\beta,T}$, we have
\begin{align*}
  \big[S_{j, j+1}& \mathcal{C}_{(j,j+1) \circ\mu(\ell), \ell} f^{(\ell)}\big] (X_{\ell-1}, V_{\ell-1})
  = \big[S_{j, j+1} \mathcal{C}_{j+1, \ell}f^{(\ell)}\big] (X_{\ell-1}, V_{\ell-1}) \\
 & = \big[ \mathcal{C}_{j+1, \ell}f^{(\ell)}\big] (x_1, \dots, x_{j+1},  x_{j}, \dots, x_{\ell-1}; \,\, v_1, \dots, v_{j+1},  v_{j}, \dots, v_{\ell-1} )\\
   & = \int_{\R^d} \int_{\S^{d-1}} B(\sigma, v_\ell - v_{j+1}) \\
   &  \qquad  \Big(
   f^{(\ell)} (x_1, \dots, x_{j+1},  x_{j}, \dots, x_\ell; \,\, v_1, \dots, v_{j+1},   v_{j}^*, \dots, v_\ell^* )\\
   &\qquad 
    - f^{(\ell)} (x_1, \dots, x_{j+1},  x_{j}, \dots, x_\ell; \,\, v_1, \dots, v_{j+1},   v_{j}, \dots, v_\ell )
   \Big) \, d\sigma dv_\ell\\
& =  \int_{\R^d} \int_{\S^{d-1}} B(\sigma, v_\ell - v_{j+1}) \\
& \qquad \Big( [S_{j, j+1}f^{(\ell)}](x_1, \dots, x_{j},  x_{j+1}, \dots, x_\ell; \,\, v_1, \dots, v_{j}^*,  v_{j+1}, \dots, v_\ell^* )\\
&\qquad 
    -[S_{j, j+1}f^{(\ell)}](x_1, \dots, x_{j},  x_{j+1}, \dots, x_\ell; \,\, v_1, \dots, v_{j}, v_{j+1},  \dots, v_\ell )
   \Big) \, d\sigma dv_\ell\\
& = \big[ \mathcal{C}_{j+1, \ell} S_{j, j+1}f^{(\ell)}\big](X_{\ell-1}, V_{\ell-1}),
\end{align*}
which proves the claim in case 2.
\end{proof}

In order to complete the proof of the invariance property \eqref{invariance}, it suffices to show
\begin{equation}
    T_{j-1}^{a - b} \mathcal{C}_{\alpha,j} T_{j}^{b-c} \mathcal{C}_{\beta,j+1} T_{j+1}^{ c- d} = T_{j-1}^{a - c} \mathcal{C}_{\beta, j} T_j^{c-b} \mathcal{C}_{\alpha, j+1} T_{j+1}^{ b -d} S_{j,j+1}.
\end{equation}
Indeed, once \eqref{three Ts} is established, it can be applied to $\mathcal{I}_{n,k}(\mu',\sigma')$ in \eqref{I'} to obtain
\begin{align*}
    \mathcal{I}_{n,k}(\mu',\sigma')f^{(n+k)}
    & =\int_{t_k \geq t_{k+1} \geq t_{k+2} \geq \dots \geq t_{k+n}\geq 0}
    \,\, T_k^{- t_{\sigma^{-1}(k+1)}} \mathcal{C}_{ \mu(k+1),k+1} 
    \, T_{k+1}^{t_{\sigma^{-1}(k+1)}- t_{\sigma^{-1}(k+2)}} \dots \nonumber \\
  & \quad \dots T_{j-1}^{t_{\sigma^{-1}(j-1)} - t_{\sigma^{-1}(j)}}
    \mathcal{C}_{ \mu(j) ,j } 
    \,\, T_{j}^{t_{\sigma^{-1}(j)}- t_{\sigma^{-1}(j+1)}} 
    \mathcal{C}_{ \mu(j+1),j+1}   \,\, T_{j+1}^{t_{\sigma^{-1}(j+1)} -t_{\sigma^{-1}(j+2)}} S_{j,j+1}
  \dots \nonumber \\
  & \quad\dots  T^{t_{\sigma^{-1}(k+n-1)} - t_{\sigma^{-1}(k+n)}}_{k+n-1} \mathcal{C}_{(j,j+1) \circ\mu(k+n), k+n}
  \, f^{(k+n)}(t_{\sigma^{-1}(k+n)}) 
  \, dt_{k+n}  \dots dt_{k+1}
\end{align*}
Then, applying identity \eqref{claim} iteratively, together with the fact that $S_{j,j+1}$ commutes with translation operators $T^\tau_\ell$ for $\ell >j+1$, and that  $f^{(k+n)}$ is symmetric i.e. $S_{j,j+1}f^{(k+n)} = f^{(k+n)}$ yields: 
\begin{align*}
    \mathcal{I}_{n,k}(\mu',\sigma')f^{(n+k)}
    & =\int_{t_k \geq t_{k+1} \geq t_{k+2} \geq \dots \geq t_{k+n}\geq 0}
    \,\, T_k^{- t_{\sigma^{-1}(k+1)}} \mathcal{C}_{ \mu(k+1),k+1} 
    \, T_{k+1}^{t_{\sigma^{-1}(k+1)}- t_{\sigma^{-1}(k+2)}} \dots \nonumber \\
  & \quad \dots T_{j-1}^{t_{\sigma^{-1}(j-1)} - t_{\sigma^{-1}(j)}}
    \mathcal{C}_{ \mu(j) ,j } 
    \,\, T_{j}^{t_{\sigma^{-1}(j)}- t_{\sigma^{-1}(j+1)}} 
    \mathcal{C}_{ \mu(j+1),j+1}   \,\, T_{j+1}^{t_{\sigma^{-1}(j+1)} -t_{\sigma^{-1}(j+2)}} 
  \dots \nonumber \\
  & \quad\dots  T^{t_{\sigma^{-1}(k+n-1)} - t_{\sigma^{-1}(k+n)}}_{k+n-1} \mathcal{C}_{\mu(k+n), k+n}
  \, f^{(k+n)}(t_{\sigma^{-1}(k+n)}) 
  \, dt_{k+n}  \dots dt_{k+1} \\
  &= \mathcal{I}_{n,k}(\mu, \sigma).
\end{align*}

It remains to prove identity \eqref{three Ts}, which is done in the following lemma.
\end{proof}

\begin{lemma}\label{main lemma}\label{three Ts lemma}
     Let $j>2$, and let $S_{j, j+1}$ be an operator that exchanges $x$ variables in positions $j$ and $j+1$, and exchanges $v$ variables in positions $j$ and $j+1$. Then any  $a,b,c,d \ge 0$ and any $\beta <\alpha<j$, we have
    \begin{equation}\label{three Ts}
    T_{j-1}^{a - b} \mathcal{C}_{\alpha,j} T_{j}^{b-c} \mathcal{C}_{\beta,j+1} T_{j+1}^{ c- d} = T_{j-1}^{a - c} \mathcal{C}_{\beta, j} T_j^{c-b} \mathcal{C}_{\alpha, j+1} T_{j+1}^{ b -d} S_{j,j+1}.
\end{equation}
\end{lemma}

\begin{proof}[Proof of Lemma \ref{main lemma}]

Note that by applying the operator $T_{j-1}^{b-a}$ from the left,  the identity \eqref{three Ts} is equivalent to
\begin{align}\label{equiv-invar}
    \mathcal{C}_{\alpha,j} T_{j}^{b-c} \mathcal{C}_{\beta,j+1} T_{j+1}^{ c- d} = T_{j-1}^{b - c} \mathcal{C}_{\beta, j} T_j^{c-b} \mathcal{C}_{\alpha, j+1} T_{j+1}^{ b -d} S_{j,j+1}
\end{align}
Furthermore, note that $S_{j,j+1}$ commutes with $T_{j+1}^\tau$ for any $\tau \in \R$. Namely, for any $f^{(j+1)}\in X^{j+1}_{p,q,\alpha,\beta,T}$, we have
\begin{align*}
    \Big[T_{j+1}^\tau & S_{j,j+1} f^{(j+1)}\Big](X_{j+1}; V_{j+1}) 
    = \Big[S_{j,j+1} f^{(j+1)}\Big] (X_{j+1} - \tau V_{j+1}; V_{j+1}) \\
    & = f^{(j+1)} (X_{j-1} - \tau V_{j-1}, x_{j+1} -\tau v_{j+1}, x_{j} -\tau v_{j}; V_{j-1}, v_{j+1}, v_j)\\
    & = \Big[T_{j+1}^\tau f^{(j+1)}\Big](X_{j-1}, x_{j+1}, x_{j};  V_{j-1}, v_{j+1}, v_{j})
    = \Big[S_{j,j+1} T_{j+1}^\tau f^{(j+1)}\Big] (X_{j+1}; V_{j+1}).
\end{align*}
Due to this commutative property, \eqref{equiv-invar} is equivalent to 
\begin{align*}
    \mathcal{C}_{\alpha,j} T_{j}^{b-c} \mathcal{C}_{\beta,j+1} T_{j+1}^{ c- d} = T_{j-1}^{b - c} \mathcal{C}_{\beta, j} T_j^{c-b} \mathcal{C}_{\alpha, j+1} S_{j,j+1} T_{j+1}^{ b -d},
\end{align*}
and by applying the operator $T_{j+1}^{d-b}$ from the right and $T_{j-1}^{c-b}$ from the left, this is equivalent to
\begin{align}\label{equivalent lemma}
    T_{j-1}^{c - b}\mathcal{C}_{\alpha,j} T_{j}^{b-c} \mathcal{C}_{\beta,j+1} T_{j+1}^{ c- b} 
    =  \mathcal{C}_{\beta, j} T_j^{c-b} \mathcal{C}_{\alpha, j+1} S_{j,j+1}.
\end{align}
So, it remains to show that \eqref{equivalent lemma} holds. Since the quantity $b-c$ is showing up in each translation operator, let us introduce an abbreviated notation
\begin{align*}
    \tau := b-c.
\end{align*}
Then \eqref{equivalent lemma} reads
\begin{align} \label{equivalent lemma tau}
     T_{j-1}^{-\tau}\mathcal{C}_{\alpha,j} T_{j}^{\tau} \mathcal{C}_{\beta,j+1} T_{j+1}^{-\tau} 
    =  \mathcal{C}_{\beta, j} T_j^{-\tau} \mathcal{C}_{\alpha, j+1} S_{j,j+1}.
\end{align}
Each $\mathcal{C}$ operator in the line above will have a velocity corresponding to $v_{k+1}^*$ in the definition \eqref{gain operator}. In order to distinguish these velocities that correspond to different $\mathcal{C}$ operators, we will denote such velocity with a prime instead of a star for $\mathcal{C}_{\beta,j+1}$  (see \eqref{prime}); for $\mathcal{C}_{\beta,j}$ we will use a tilde instead of a star (see \eqref{tilde}), and for $\mathcal{C}_{\alpha,j+1}$ we will use a sharp instead of a star (see \eqref{sharp}). Prime, tilde and sharp notation will be used only within this lemma.

In order to prove \eqref{equivalent lemma tau}, we expand its left-hand side (LHS)  by first applying operators $T^{-\tau}_{j-1}$ and  $\mathcal{C}_{\alpha, j}$:
\begin{align*}
\text{LHS}& =   \Big[T^{-\tau}_{j-1}   \mathcal{C}_{\alpha, j} T^{\tau}_j \mathcal{C}_{\beta, j+1} T^{-\tau}_{j+1} \, f^{(j+1)}\Big](X_{j-1}, V_{j-1}) \\
  &  =   \Big[\mathcal{C}_{\alpha, j} T^{\tau}_j \mathcal{C}_{\beta, j+1} T^{-\tau}_{j+1}\, f^{(j+1)}\Big]( X_{j-1} + \tau V_{j-1}, \,  V_{j-1}) \\
  & =  \int_{\R^d} \int_{\S^{d-1}} B(\sigma, v_{j} - v_\alpha) 
    \Big( \Big[T^{\tau}_j \mathcal{C}_{\beta, j+1} T^{-\tau}_{j+1}\, f^{(j+1)}\Big]
        ( X_{j-1} + \tau  V_{j-1}, \,x_\alpha;
        \,\,  V_{j-1}^{*\alpha}, v^*_j) \\
    & \hspace{4cm} -  \Big[T^{\tau}_j \mathcal{C}_{\beta, j+1} T^{-\tau}_{j+1}\, f^{(j+1)}\Big]  ( X_{j-1} +\tau V_{j-1}, x_\alpha;\,\,  V_{j-1}, v_j) \Big) d \, \sigma dv_j,
  \end{align*}
where by \eqref{*j} and \eqref{*k+1},
\begin{align*}
& V_{j-1}^{*\alpha} = (v_1,\dots, v_\alpha^*, \dots, v_{j-1}), \\
& v_\alpha^*  = \frac{v_\alpha+v_{j}}{2}+\frac{|v_{j}-v_\alpha|}{2}\sigma, \\
 &  v_j^*=\frac{v_\alpha+v_{j}}{2}-\frac{|v_{j}-v_\alpha|}{2}\sigma.
\end{align*}

We then apply the operator $T_j^{\tau}$ to obtain
\begin{align*}
  \text{LHS}
& =  \int_{\R^d} \int_{\S^{d-1}} \, B(\sigma, v_{j} - v_\alpha) 
    \Big( \Big[\mathcal{C}_{\beta, j+1} T^{-\tau}_{j+1}\, f^{(j+1)}\Big]
        ( X_{j-1} +\tau (V_{j-1} -V_{j-1}^{*\alpha}), x_\alpha- \tau v_j^*; 
        \,\,  V_{j-1}^{*\alpha}, v^*_j) \\
    & \hspace{4cm} -  \Big[\mathcal{C}_{\beta, j+1} T^{-\tau}_{j+1}\, f^{(j+1)}\Big]  ( X_{j-1}, x_\alpha-\tau v_j; \,\,  V_{j-1}, v_j) \Big)  d\sigma dv_j.
\end{align*}
  Next, we apply the operator $\C_{\beta, j+1}$
 to obtain
 \begin{align*}
\text{LHS} & = \int_{\R^d} \int_{\S^{d-1}}\int_{\R^d} \int_{\S^{d-1}} \, d \sigma dv_j d \sigma' dv_{j+1} \, B(\sigma, v_{j} - v_\alpha) \,B(\sigma', v_{j+1} - v_\beta)\\
& \Big( \big[T^{-\tau}_{j+1}f^{(j+1)}\big]
    (X_{j-1} +\tau (V_{j-1} -V_{j-1}^{*\alpha}), 
    x_\alpha - \tau v_j^*, x_\beta; 
    \,\, V_{j-1}^{*\alpha, '\beta}, v^*_j, v_{j+1}') \\
& - \big[T^{-\tau}_{j+1}f^{(j+1)}\big]
    (X_{j-1} +\tau (V_{j-1} -V_{j-1}^{*\alpha}), 
    x_\alpha- \tau v_j^*, x_\beta;
    \,\, V_{j-1}^{*\alpha}, v^*_j, v_{j+1})\\
& - \big[T^{-\tau}_{j+1}f^{(j+1)}\big]
    (X_{j-1}, x_\alpha- \tau v_j, x_\beta; 
    \,\,  V^{'\beta}_{j-1}, v_j, v'_{j+1})\\
& + \big[T^{-\tau}_{j+1}f^{(j+1)}\big]
    (X_{j-1}, x_\alpha- \tau v_j, x_\beta; 
    \,\,  V_{j-1}, v_j, v_{j+1})
\Big),
 \end{align*}
 where, by \eqref{*j} and \eqref{*k+1}, since $\beta < \alpha < j$, we have
 \begin{align}\label{prime}
&  V_{j-1}^{*\alpha, '\beta} 
    : = (v_1, \dots, v'_\beta, \dots, v_\alpha^*, \dots, v_{j-1}), \nonumber\\
    & v'_\beta  = \frac{v_\beta +v_{j+1}}{2}+\frac{|v_{j+1}-v_\beta|}{2}\sigma', \\
 &  v'_{j+1}=\frac{v_\beta+v_{j+1}}{2}-\frac{|v_{j+1}-v_\beta|}{2}\sigma'. \nonumber
 \end{align}
Finally, we apply the operator $T_{j+1}^{-\tau}$ to obtain
 \begin{align*}
\text{LHS} & = \int_{\R^d} \int_{\S^{d-1}}\int_{\R^d} \int_{\S^{d-1}} \, d \sigma dv_j d \sigma' dv_{j+1} \, B(\sigma, v_{j} - v_\alpha) \,B(\sigma', v_{j+1} - v_\beta) \nonumber\\
& \Big( f^{(j+1)}
    (X_{j-1} +\tau (V_{j-1} -V_{j-1}^{*\alpha} + V_{j-1}^{*\alpha, '\beta}), 
    x_\alpha, x_\beta + \tau v_{j+1}'; 
    \,\, V_{j-1}^{*\alpha, '\beta}, v^*_j, v_{j+1}') \nonumber \\
& - f^{(j+1)}
    (X_{j-1} +\tau V_{j-1}, 
    x_\alpha, x_\beta +\tau v_{j+1};
    \,\, V_{j-1}^{*\alpha}, v^*_j, v_{j+1}) \nonumber\\
& - f^{(j+1)}
    (X_{j-1} + \tau  V^{'\beta}_{j-1}, x_\alpha, x_\beta +\tau v'_{j+1}; 
    \,\,  V^{'\beta}_{j-1}, v_j, v'_{j+1}) \nonumber \\
& + f^{(j+1)}
    (X_{j-1} +\tau V_{j-1}, x_\alpha, x_\beta + \tau v_{j+1}; 
    \,\,  V_{j-1}, v_j, v_{j+1})
\Big).
 \end{align*}
 Note that
 \begin{align*}
     V_{j-1} -V_{j-1}^{*\alpha} + V_{j-1}^{*\alpha, '\beta} = V_{j-1}^{'\beta}, 
 \end{align*}
and therefore,
 \begin{align}
\text{LHS} & = \int_{\R^d} \int_{\S^{d-1}}\int_{\R^d} \int_{\S^{d-1}} \, d \sigma dv_j d \sigma' dv_{j+1} \, B(\sigma, v_{j} - v_\alpha) \,B(\sigma', v_{j+1} - v_\beta) \nonumber\\
& \Big( f^{(j+1)}
    (X_{j-1} +\tau V_{j-1}^{'\beta}, 
    x_\alpha, x_\beta + \tau v_{j+1}'; 
    \,\, V_{j-1}^{*\alpha, '\beta}, v^*_j, v_{j+1}') \nonumber \\
& - f^{(j+1)}
    (X_{j-1} +\tau V_{j-1}, 
    x_\alpha, x_\beta +\tau v_{j+1};
    \,\, V_{j-1}^{*\alpha}, v^*_j, v_{j+1}) \nonumber\\
& - f^{(j+1)}
    (X_{j-1} + \tau  V^{'\beta}_{j-1}, x_\alpha, x_\beta +\tau v'_{j+1}; 
    \,\,  V^{'\beta}_{j-1}, v_j, v'_{j+1}) \nonumber \\
& + f^{(j+1)}
    (X_{j-1} +\tau V_{j-1}, x_\alpha, x_\beta + \tau v_{j+1}; 
    \,\,  V_{j-1}, v_j, v_{j+1})
\Big). \label{LHS}
 \end{align}

Next, we expand the right-hand side (RHS) of \eqref{equivalent lemma tau} by first applying the operator $ \mathcal{C}_{\beta, j}$:
\begin{align*}
 \text{RHS} 
 & = \Big[\mathcal{C}_{\beta, j} T_j^{-\tau} \mathcal{C}_{\alpha, j+1} S_{j,j+1} f^{(j+1)}\Big](X_{j-1}, V_{j-1}) \\
 & = \int_{\R^d} \int_{\S^{d-1}} \,B(\tild{\sigma}, v_j - v_{\beta})
 \Bigg(
        \Big[T_j^{-\tau} \mathcal{C}_{\alpha, j+1} S_{j,j+1} f^{(j+1)}\Big] 
        (X_{j-1}, x_\beta; \, V_{j-1}^{\,\,\tild{} \,\beta}, \tild v_j) \\
& \hspace{3.5cm}
        - \Big[T_j^{-\tau} \mathcal{C}_{\alpha, j+1} S_{j,j+1} f^{(j+1)}\Big] 
        (X_{j-1}, x_\beta; \, V_{j-1}, v_j)
    \Bigg)  d\tild{\sigma} dv_j \, ,
\end{align*}
where, according to \eqref{*j} and \eqref{*k+1}, we have
\begin{align}
& V_{j-1}^{\,\, \tild{} \, \beta} = (v_1, \dots, \tild v_\beta, \dots, v_{j-1}), \nonumber \\
& \tild v_\beta =  \frac{v_\beta +v_{j}}{2}+\frac{|v_{j}-v_\beta|}{2}\tild\sigma, \label{tilde}\\
& \tild v_j = \frac{v_\beta +v_{j}}{2}-\frac{|v_{j}-v_\beta|}{2}\tild\sigma. \nonumber
\end{align}
Then we apply operator $T_j^{-\tau}$ 
\begin{align*}
 \text{RHS} 
 & = \int_{\R^d} \int_{\S^{d-1}} \, d\tild \sigma dv_j \,B(\tild \sigma, v_j - v_{\beta})\\
 & \hspace{2cm} 
    \times \Bigg(
    \Big[\mathcal{C}_{\alpha, j+1} S_{j,j+1} f^{(j+1)}\Big] 
        (X_{j-1}+\tau V_{j-1}^{\,\, \tild{} \,\beta} , x_\beta +\tau \tild v_j; \,\, V_{j-1}^{\,\, \tild{} \,\beta}, \tild v_j) \\
& \hspace{3cm}
    - \Big[\mathcal{C}_{\alpha, j+1} S_{j,j+1} f^{(j+1)}\Big] 
        (X_{j-1}+\tau V_{j-1} , x_\beta + \tau v_j; \,\, V_{j-1}, v_j)
    \Bigg)   ,
\end{align*}
and then operator $ \mathcal{C}_{\alpha, j+1}$
\begin{align*}
 \text{RHS} 
 & = 
 \int_{\R^d} \int_{\S^{d-1}} \int_{\R^d} \int_{\S^{d-1}} \, d\tild \sigma dv_j d \sigma^\# dv_{j+1}\,B(\tild \sigma, v_j - v_{\beta}) B(\sigma^\#, v_{j+1} - v_{\alpha})\\
 & \hspace{2cm} \times
 \Big(
     \Big[S_{j,j+1} f^{(j+1)}\Big]
     (X_{j-1}+\tau V_{j-1}^{\,\, \tild{} \, \beta} , x_\beta +\tau \tild v_j, x_\alpha; \,\, V_{j-1}^{\,\, \tild{} \, \beta, \,\# \,\alpha}, \tild v_j,  v^\#_{j+1})\\
 & \hspace{2.5cm} 
   -  \Big[S_{j,j+1} f^{(j+1)}\Big] (X_{j-1}+\tau V_{j-1}^{\,\,\tild{} \,\beta} , x_\beta +\tau \tild v_j, x_\alpha; \,\, V_{j-1}^{\,\, \tild{} \,\beta}, \tild v_j, v_{j+1})\\
 & \hspace{2.5cm} 
   -  \Big[S_{j,j+1} f^{(j+1)}\Big] (X_{j-1}+\tau V_{j-1} , x_\beta + \tau v_j, x_\alpha; \,\, V^{\# \alpha}_{j-1}, v_j, v^\#_{j+1})\\
& \hspace{2.5cm} 
   +  \Big[S_{j,j+1} f^{(j+1)}\Big] (X_{j-1}+\tau V_{j-1} , x_\beta + \tau v_j, x_\alpha; \,\, V_{j-1}, v_j, v_{j+1})
 \Big),
 \end{align*}
where
\begin{align}
& V_{j-1}^{\,\, \tild{} \, \beta, \,\# \,\alpha} =  (v_1, \dots, \tild v_\beta, \dots, v^\#_\alpha, \dots, v_{j-1}), \nonumber\\
& V^{\#\alpha}_{j-1} = (v_1, \dots, v_\beta, \dots,  v^\#_\alpha, \dots, v_{j-1}), \label{sharp}\\
& v^\#_\alpha = \frac{v_\alpha +v_{j+1}}{2}+\frac{|v_{j+1}-v_\alpha|}{2}\sigma^\#, \nonumber\\
& v^\#_{j+1} = \frac{v_\alpha +v_{j+1}}{2}-\frac{|v_{j+1}-v_\alpha|}{2}\sigma^\#. \nonumber
\end{align}
Next we apply operator $S_{j, j+1}$ to obtain:
 \begin{align*}
 \text{RHS} 
 & = 
 \int_{\R^d} \int_{\S^{d-1}} \int_{\R^d} \int_{\S^{d-1}} \, d\tild \sigma dv_j d\sigma^\# dv_{j+1}\,B(\tild \sigma, v_j - v_{\beta}) B(\sigma^\#, v_{j+1} - v_{\alpha})\\
 & \hspace{2cm} \times
 \Big(
      f^{(j+1)}
    (X_{j-1}+\tau V_{j-1}^{\,\, \tild{} \, \beta}, x_\alpha, x_\beta +\tau \tild v_j; \,\, V_{j-1}^{\,\, \tild{} \, \beta, \,\# \,\alpha} ,  v^\#_{j+1}, \tild v_j)\\
 & \hspace{2.5cm} 
   -  f^{(j+1)} (X_{j-1}+\tau V_{j-1}^{\,\,\tild{} \,\beta}, x_\alpha , x_\beta +\tau \tild v_j; \,\, V_{j-1}^{\,\, \tild{} \,\beta} , v_{j+1}, \tild v_j)\\
 & \hspace{2.5cm} 
   -   f^{(j+1)} (X_{j-1}+\tau V_{j-1} , x_\alpha , x_\beta + \tau v_j; \,\, V^{\# \alpha}_{j-1} , v^\#_{j+1}, v_j) \\
& \hspace{2.5cm} 
   +  f^{(j+1)} (X_{j-1}+\tau V_{j-1} , x_\alpha , x_\beta + \tau v_j; \,\, V_{j-1}, v_{j+1}, v_j)
 \Big).
 \end{align*}
Finally, we apply changes of variables $v_j \leftrightarrow v_{j+1}$, $\tild \sigma \mapsto \sigma'$ and $\sigma^\#\mapsto \sigma$, and note that under such changes of variables we have
\begin{align*}
    & \tild v_\beta \mapsto  \frac{v_\beta +v_{j+1}}{2}+\frac{|v_{j+1}-v_\beta|}{2}\sigma' = v'_\beta,\\
    & \tild v_j \mapsto \frac{v_\beta +v_{j+1}}{2} - \frac{|v_{j+1}-v_\beta|}{2}\sigma' = v'_{j+1},\\
    & v^\#_\alpha \mapsto  \frac{v_\alpha +v_{j}}{2}+\frac{|v_{j}-v_\alpha|}{2}\sigma = v^*_\alpha,\\
    & v^\#_{j+1} \mapsto \frac{v_\alpha +v_{j}}{2}-\frac{|v_{j}-v_\alpha|}{2}\sigma = v_j^*.
\end{align*}
Therefore,
\begin{align*}
 \text{RHS} 
 & = 
 \int_{\R^d} \int_{\S^{d-1}} \int_{\R^d} \int_{\S^{d-1}} \,  d\sigma' dv_{j+1} d\sigma dv_{j}\,B(\sigma', v_{j+1} - v_{\beta}) B(\sigma, v_{j} - v_{\alpha})\\
 & \hspace{2cm} \times
 \Big(
      f^{(j+1)}
    (X_{j-1}+\tau V_{j-1}^{'  \beta}, x_\alpha, x_\beta +\tau  v'_{j+1}
        ; \,\, V_{j-1}^{'  \beta, \,*\alpha} ,  v^*_{j},  v'_{j+1})\\
 & \hspace{2.5cm} 
   -  f^{(j+1)} (X_{j-1}+\tau V_{j-1}^{' \,\beta}, x_\alpha , x_\beta +\tau  v'_{j+1}
    ; \,\, V_{j-1}^{' \,\beta} , v_{j}, v'_{j+1})\\
 & \hspace{2.5cm} 
   -   f^{(j+1)} (X_{j-1}+\tau V_{j-1} , x_\alpha , x_\beta + \tau v_j
        ; \,\, V^{* \alpha}_{j-1} , v^*_{j}, v_{j+1}) \\
& \hspace{2.5cm} 
   +  f^{(j+1)} (X_{j-1}+\tau V_{j-1} , x_\alpha , x_\beta + \tau v_j
        ; \,\, V_{j-1}, v_{j}, v_{j+1})
 \Big).
 \end{align*}

 Compare that with the formula for $\text{LHS}$ in \eqref{LHS}
 \begin{align*}
\text{LHS} & = \int_{\R^d} \int_{\S^{d-1}}\int_{\R^d} \int_{\S^{d-1}} \, d \sigma dv_j d \sigma' dv_{j+1} \, B(\sigma, v_{j} - v_\alpha) \,B(\sigma', v_{j+1} - v_\beta) \nonumber\\
& \Big( f^{(j+1)}
    (X_{j-1} +\tau V_{j-1}^{'\beta}, 
    x_\alpha, x_\beta + \tau v_{j+1}'; 
    \,\, V_{j-1}^{*\alpha, '\beta}, v^*_j, v_{j+1}') \nonumber \\
& - f^{(j+1)}
    (X_{j-1} +\tau V_{j-1}, 
    x_\alpha, x_\beta +\tau v_{j+1};
    \,\, V_{j-1}^{*\alpha}, v^*_j, v_{j+1}) \nonumber\\
& - f^{(j+1)}
    (X_{j-1} + \tau  V^{'\beta}_{j-1}, x_\alpha, x_\beta +\tau v'_{j+1}; 
    \,\,  V^{'\beta}_{j-1}, v_j, v'_{j+1}) \nonumber \\
& + f^{(j+1)}
    (X_{j-1} +\tau V_{j-1}, x_\alpha, x_\beta + \tau v_{j+1}; 
    \,\,  V_{j-1}, v_j, v_{j+1})
\Big)
 \end{align*}
to conclude that the left-hand side and the right-hand side of \eqref{equivalent lemma tau} are indeed equal. This concludes the proof of Lemma \ref{main lemma}. 
\end{proof}

Motivated by \cite{klma08}, we give the following definition. 
\begin{definition}[Special Upper Echelon Form]
We say that $\mu \in M_n$ is in special upper echelon form if for every $j\in \{k+1, \dots, k+n\}$ we have $\mu(j) \leq \mu(j+1)$. We will denote $\mathcal{M}_n$ to be the set of all special upper echelon forms in $M_n$.
\end{definition}

\begin{proposition}[Combinatorics, \protect{\cite[Lemma 3.2, Lemma 3.3]{klma08}}] \label{proposition-combinatorics}
For any $n \in \N$, the following statements are true:
\begin{itemize}
    \item[(i)] Any $\mu \in M_n$ can be changed to special upper echelon form via a finite sequence of acceptable moves.
    \item[(ii)] The following upper bound holds: $\# \mathcal{M}_n \leq 2^{k+n}$. 
\end{itemize}
\end{proposition}
The proof of both of these facts is the content of Lemmas 3.2, 3.3 in \cite{klma08}.

In each equivalence class, we can now reorganize the sum in the decomposition of $f^{(k)}$. 
\begin{proposition}[Sum Over Special Upper Echelon Forms]\label{sum over echelon forms}
Let $\mu_s \in \mathcal{M}_n$ be a special upper echelon form, and write $\mu \sim \mu_s$ if $\mu$ can be reduced to $\mu_s$ in finitely many acceptable moves. There exists a set $D \subset [0,t_k]^{n}$ such that 
\begin{equation}
    \sum_{\mu \sim \mu_s} \int_{t_{k} \geq \dots \geq t_{k+n} \geq 0} J(\underline{t}_{n,k} ; \mu) \, dt_{k+n} \dots  d t_{k+1} 
        = \int_{D} J(\underline{t}_{n,k}; \mu_s) \, dt_{k+n} \dots d t_{k+1},
\end{equation}
where the sum occurs over $\mu \in M_n$ such that $\mu$ can be changed to $\mu_s$ via a sequence of acceptable moves.
\end{proposition}
The proof of this proposition is identical to that of Theorem 3.4 in \cite{klma08}.

\subsection{A priori estimate and combinatorial reorganization in action together}\label{estimate+comb}
We now prove  Theorem \ref{Proposition with 0 initial data}.
 Recall that the sequence $F=(f^{(k)})_{k=1}^\infty$ is a mild $\mu$-solution of the Boltzmann hierarchy \eqref{BH} corresponding to $F_0=0$ if for every $k \in \N$ the formula \eqref{non-iterated} holds.
Additionally, recall the iterated formula for $ T^{-t}_kf^{(k)}(t)$ in \eqref{iteration1}, which with the help of notation introduced in \eqref{J}-\eqref{J mu}, can be written as:
\begin{align*}
T^{-t}_k f^{(k)}(t)
	 &=  \sum_{\mu \in M_n}  \int_0^{t}\int_{0}^{t_{k+1}} \cdots \int_0^{t_{k+n-1}} J_{n,k}(\underline{t}_n; \mu) \,\,  dt_{k+n} \cdots  dt_{k+1}.
\end{align*}
By Proposition \ref{sum over echelon forms}, one can instead sum over all equivalence classes as follows:
\begin{align*}
T^{-t}_k f^{(k)}(t)
	 &=  \sum_{\mu_{s} \in \mathcal{M}_n}   \int_{D(\mu_s)}  J_{n,k}(\underline{t}_n; \mu(s)) \,\,  dt_{k+n} \cdots  dt_{k+1}.
\end{align*}
Recalling the definition of $J_n$ in \eqref{J mu} and the fact that each $\C^{k+\ell}$ is a sum $\C^{k+\ell} = \sum_{j=1}^{k+\ell-1} (\C^+_{j, k+\ell} - \C^-_{j, k+\ell})$, we obtain:
\begin{align}\label{after board game}
    T_k^{-t} f^{(k)}(t)& =  \sum_{\mu_{s} \in \mathcal{M}_n}\sum_{\bm{\pi}\in \{+,-\}^n} |\bm{\pi}| \int_{D(\mu_s)} 
   T_k^{- t_{k+1}} \mathcal{C}_{\mu(k+1),k+1} ^{\pi_1}T_{k+1}^{t_{k+1}- t_{k+2}} \mathcal{C}_{\mu(k+2),k+2}^{\pi_2} \nonumber \\
    & \qquad \qquad \qquad\cdots T^{t_{k+n-1} - t_{k+n}}_{k+n-1} \mathcal{C}_{\mu(k+n), k+n}^{\pi_n}
    f^{(k+n)}(t_{k+n}) \, dt_{k+n} \dots dt_{k+1},
\end{align}
where for  $\bm{\pi} = (\pi_1, \dots, \pi_n) \in \{+, -\}^n$, we define $|\bm{\pi}|=\prod_{\ell=1}^n\pi_\ell$.

Since, 
 $|\C^\pm_{j,k} g^{(k)}| \le \C^\pm_{j,k} |g^{(k)}|$  and $|T_k^{s}g^{(k)}| = T_k^{s}|g^{(k)}|$, 
 we have
 \begin{align}
    \left|T_k^{-t} f^{(k)}(t)\right|& \le  \sum_{\mu_{s} \in \mathcal{M}_n}\sum_{\bm{\pi}\in \{+,-\}^n} \int_{D(\mu_s)} 
   T_k^{- t_{k+1}} \mathcal{C}_{\mu(k+1),k+1} ^{\pi_1}T_{k+1}^{t_{k+1}- t_{k+2}} \mathcal{C}_{\mu(k+2),k+2}^{\pi_2} \nonumber \\
    & \qquad \qquad \qquad\cdots T^{t_{k+n-1} - t_{k+n}}_{k+n-1} \mathcal{C}_{\mu(k+n), k+n}^{\pi_n}
    \left|f^{(k+n)}\right|(t_{k+n}) dt_{k+n} \dots dt_{k+1}.
\end{align}
Now, due to the non-negativity of  the function $|f^{(k+n)}|$, this can further be estimated by enlarging the domain of time integration as follows:
\begin{align*}
\left|T_k^{-t} f^{(k)}(t)\right|
    & \le  \sum_{\mu_{s} \in \mathcal{M}_n}\sum_{\bm{\pi}\in \{+,-\}^n} \int_{[0,T]^n} 
   T_k^{- t_{k+1}} \mathcal{C}_{\mu(k+1),k+1} ^{\pi_1}T_{k+1}^{t_{k+1}- t_{k+2}} \mathcal{C}_{\mu(k+2),k+2}^{\pi_2} \nonumber \\
    & \hspace{4cm} \cdots T^{t_{k+n-1} - t_{k+n}}_{k+n-1} \mathcal{C}_{\mu(k+n), k+n}^{\pi_n}
    \left|f^{(k+n)}\right|(t_{k+n}) dt_{k+n} \dots dt_{k+1}.
\end{align*}
 Multiplying both sides of the above inequality with $\l\l \alpha X_k\r\r^p\l\l \beta V_k\r\r^q$, taking supremum in $X_k,V_k$, and using the triangle inequality on the norm $\|\cdot\|_{k,p,q,\alpha,\beta}$, we obtain
\begin{equation*}
\begin{aligned}
\left\|T_k^{-t} f^{(k)}(t)\right\|_{k,p,q,\alpha,\beta}
&\le 
\sum_{\mu_s\in\mathcal{M}_n}\sum_{\bm{\pi}\in\{+,-\}^n} \left\| \int_{[0,T]^n}
   T_k^{- t_{1}} \mathcal{C}_{\mu(k+1),k+1}^{\pi_1} T_{k+1}^{t_{1}- t_{2}} \mathcal{C}_{\mu(k+2),k+2}^{\pi_2}  \right.\\
   &\hspace{3cm} \cdots T^{t_{n-1} - t_{n}}_{k+n-1} \mathcal{C}_{\mu(k+n), k+n}^{\pi_n}
    \left|f^{(k+n)}\right|(t_{n}) dt_{k+n} \dots dt_{k+1}\Bigg\|_{k,p,q,\alpha,\beta}.
\end{aligned} 
\end{equation*}
Applying Corollary \ref{a-priori estimate step n}, the estimate on the number of equivalence classes in part (ii) of Proposition \ref{proposition-combinatorics}, the definition of norm \eqref{full sequence norm} yields
\begin{equation}
\begin{aligned}
\left\|T_k^{-t} f^{(k)}(t)\right\|_{k,p,q,\alpha,\beta}&\le  2^{k+2n} C_{p,q,\alpha,\beta}^n \left|\left|\left|T_{k+n}^{-(\cdot)}\left|f^{(k+n)}\right|(\cdot)\right|\right|\right|_{k+n,p,q,\alpha,\beta,T}\\
&= 2^{k+2n} C_{p,q,\alpha,\beta}^n
    \left|\left|\left|
    T_{k+n}^{-(\cdot)}f^{(k+n)}(\cdot)\right|\right|\right|_{k+n,p,q,\alpha,\beta,T}\\
&\le2^{k+1+2n} C_{p,q,\alpha,\beta}^n  \,e^{-\mu(k+n)}|||\mathcal{T}^{-(\cdot)}F|||_{p,q,\alpha,\beta,\mu,T}\\
&= 2(2e^{-\mu})^k (4C_{p,q,\alpha,\beta}\,e^{-\mu})^n|||\mathcal{T}^{-(\cdot)}F|||_{p,q,\alpha,\beta,\mu,T}.
\end{aligned} 
\end{equation}

Since $e^\mu>4C_{p,q,\alpha,\beta}$, and $|||\mathcal{T}^{-(\cdot)}F|||_{p,q,\alpha,\beta,\mu,T}<\infty$ we let $n\to\infty$ to obtain $\|T_k^{-t} f^{(k)}(t)\|_{k,p,q,\alpha,\beta}=0$. Since $t\in[0,T]$ was arbitrary, we obtain $T_k^{-(\cdot)} f^{(k)}(\cdot)=0$. 
Hence $f^{(k)}=0$, and thus $F=0$.
\begin{flushright}
    $\square$
\end{flushright}

\section{Existence of Solutions to the Boltzmann equation and to the Boltzmann hierarchy}\label{Existence section}
This section is devoted to the construction of a global in time solution to the Boltzmann hierarchy \eqref{BH} for space-velocity polynomially decaying initial data and  a range of values for the chemical potential. As mentioned in Section \ref{sec-intro}, part of this construction relies on solving  the Boltzmann equation itself. For this reason we first review  several things about the Boltzmann equation.
\subsection{The Boltzmann equation}\label{ssec:Boltzmann equation}
 The Boltzmann equation for a function  $f:[0,\infty)\times\R^{d}\times\R^{d}\to \R$ with initial data $f_0:\R^{d}\times\R^{d}\to \R$, is given by
\begin{equation}\label{BE}
\begin{cases}
 \partial_t f+ v\cdot\nabla_{x}f=Q(f,f),\\
 \\
 f(t=0)=f_0,
 \end{cases}
\end{equation}
where
the collisional operator (in generalized multilinear form)  is given by
\begin{equation}
Q(g,h)(t,x,v)= \int_{\R^d\times\S^{d-1}} |u|^\gamma b(\hat{u}\cdot \sigma)\left(g^*h_1^*-gh_1\right)\,d\sigma\,dv_1 ,  
\end{equation}
and we use the notation
\begin{align}
&u=v_1-v\\
&g^*:=g(t,x,v^*),\,\,h_1^*:=h(t,x,v_1^*),\,\,g:=g(t,x,v),\,\,h_1:=h(t,x,v_1) \\
    &v^*=\frac{v+v_1}{2}+\frac{|v_1-v|}{2}\sigma,\\
    &v_{1}^*=\frac{v+v_1}{2}-\frac{|v_1-v|}{2}\sigma.
\end{align}
It is well known (see e.g. \cite{ceilpu94}) that the collisional operator $Q$ can be written in weak form as follows:
\begin{equation}\label{weak form}
\int_{\R^d} Q(g,h)\phi(v)\,dv=\frac{1}{2}\int_{\mathbb{S}^{d-1}\times \R^{2d}} B(\sigma, v_1-v) \, gh_1\left(\phi(v^*)+\phi(v_1^*)-\phi(v)-\phi(v_1)\right)\,d\sigma\,dv_1\,dv,   
\end{equation}
where $\phi$ is a test function, appropriate for the above integrations to make sense. 

In particular for $\phi\in\{1,v,|v|^2\}$
the conservation of momentum and energy at the collisional level formally imply 
\begin{equation}\label{collisional averaging}
    \int_{\R^d} Q(g,h)\phi\,dv=0,
\end{equation}
which  yields  the conservation of mass, momentum and energy 
\begin{equation}
 \partial_t \int_{\R^d} f \phi\,dv=0,\quad \phi\in\{1,v,|v|^2\},   
\end{equation}
for a solution $f$ to \eqref{BE}.

\begin{remark}\label{sc 4 remark on tensorization}
A direct computation shows that $f$ formally solves \eqref{BE} with initial data $f_0$ if and only if $F=(f^{\otimes k})_{k\in\mathbb{N}}$ formally solves \eqref{BH} with initial data $F_0=(f_0^{\otimes k})_{k\in\mathbb{N}}$.
Hence, at the formal level, one can construct solutions to the Boltzmann hierarchy \eqref{BH} by tensorizing solutions of the Boltzmann equation \eqref{BE}.
\end{remark}
\subsection{Global well-posedness of the Boltzmann equation for small space-velocity polynomially decaying initial data}\label{ssec well posedness of BE}Now, we present the global well-posedness result for  the Boltzmann equation we rely on to construct global solutions to the Boltzmann hierarchy. This result was proved in \cite{to86} for $d=3$. Here we extend the result to arbitrary dimension $d\ge 3$, as well as rigorously address the conservation laws of the solutions.

Let us first give the precise definition of a mild solution to the Boltzmann equation and then state the well-posedness result. Recall notation from \eqref{abbreviation static}, \eqref{abbreviation time}.
In the same spirit as in Definition \ref{BH mild definition}, we define   mild solutions to the Boltzmann equation as follows:
\begin{definition}
Let $T>0$, $p,q>1$ and $\alpha,\beta>0$.  A measurable function $f:[0,T]\times\R^d\times\R^d\to\R$ is called a mild solution to the Boltzmann equation \eqref{BE} in $[0,T]$ corresponding to the initial data $f_0\in X_{p,q,\alpha,\beta}$ if $T_1^{-(\cdot)}f(\cdot)\in X_{p,q,\alpha,\beta,T}$ and 
\begin{equation}\label{definition of mild solution equation}
    T_1^{-t} f(t,x,v)= f_0+\int_0^t T_1^{-s}Q[f,f](s,x,v)\,ds,\quad t\in[0,T].
\end{equation}
\end{definition}

We are now in the position to state the global well-posedness result for the Boltzmann equation \eqref{BE}.
\begin{theorem} \label{BE is well posed}
Let $T>0$, $p>1$, $q>\max\{d+\gamma-1,d-1\}$ and $M>0$ with $M<(8C_{p,q,\alpha,\beta})^{-1}$, where $C_{p,q,\alpha,\beta}$ is given by \eqref{C_p,q}. Consider $f_0\in X_{p,q,\alpha,\beta}$, with $\|f_0\|_{p,q,\alpha,\beta}\le \frac{M}{2}$. Then there exists a unique mild solution to the Boltzmann equation \eqref{BE}, in the class of functions satisfying:
\begin{equation}\label{condition on uniqueness equation}
   |||T_1^{-(\cdot)}f(\cdot)|||_{p,q,\alpha,\beta,T}\leq  M. 
\end{equation}
If $f_0\ge 0$, then the solution remains non-negative.
Additionally, assuming that $f,g$ are the mild solutions corresponding to initial data $f_0,g_0$ respectively, there holds the continuity with respect to initial data estimate:
\begin{equation}\label{stability estimate}
   |||T_1^{-(\cdot)}f(\cdot)-T_1^{-(\cdot)}g(\cdot)|||_{p,q,\alpha,\beta,T}\leq 2 \|f_0-g_0\|_{p,q,\alpha,\beta}. 
\end{equation}
In particular 
\begin{equation} \label{bound wrt initial data equation}
    |||T_1^{-(\cdot)}f(\cdot)|||_{p,q, \alpha,\beta,T}\leq 2 \|f_0\|_{p,q,\alpha,\beta}.
\end{equation}
Moreover, if $\gamma\ge 0$, the solution satisfies the following conservation laws for any $t\in[0,T]$ and a.e. $x\in\R^d$:
\begin{equation}\label{conservation of mass equation}
\text{If $p>d,\,q>d+\gamma$}:\,\,\int_{\R^d}f(t,x,v)\,dv=\int_{\R^d}f_0(x,v)\,dv    
\end{equation}
\begin{equation}\label{conservation of momentum equation}
\text{If $p>d,\,q>d+\gamma+1$ }:\,\,\int_{\R^d}vf(t,x,v)\,dv=\int_{\R^d}vf_0(x,v)\,dv    
\end{equation}
\begin{equation}\label{conservation of energy equation}
\text{If $p>d,\,q>d+\gamma+2$ }:\,\,\int_{\R^d}|v|^2f(t,x,v)\,dv=\int_{\R^d}|v|^2f_0(x,v)\,dv.    
\end{equation}
\end{theorem}
\begin{proof}
 For $d=3$ and $\alpha=\beta =1$ the result has already been proved in \cite{to86}, except the conservation laws \eqref{conservation of mass equation}-\eqref{conservation of energy equation}. The proof of the result in \cite{to86} relies on Lemma \ref{lemma on positions weight} and Lemma \ref{lemma on velocities weight} which were proved for $d=3$ in \cite{beto85,to86} respectively.  Since we prove Lemma \ref{lemma on positions weight} and Lemma \ref{lemma on velocities weight} for any dimension $d\geq 3$ in the Appendix, one can follow the strategy of \cite{to86} to naturally extend the result to any dimension $d\geq 3$; therefore we omit the proof.

 It remains to prove the conservation laws \eqref{conservation of mass equation}-\eqref{conservation of energy equation}. Assume $\gamma\geq 0$ and consider $i\in\{0,1,2\}$.  Assume also that $p>d$ and $q>d+\gamma+i$. By \eqref{condition on uniqueness equation}, for any $t\in [0,T]$ and a.e. $x,v\in\R^d$, we have
\begin{equation}\label{pointwise bound}\l \beta v\r^{\gamma+i}f(t,x,v)\le M\l \alpha(x-vt)\r^{-p}\l \beta v\r^{\gamma+i-q}\leq M\l \beta v\r^{\gamma+i-q}.
\end{equation}
Since $q>d+\gamma+i$, we   integrate \eqref{pointwise bound} in velocity to  obtain that $f(t,x,v)\in L^{1,\gamma+i}_v$, for any $t\in[0,T]$ and a.e. $x\in\R^d$,
where given $\ell\geq 0$ we denote
$$L_v^{1,\ell}=\left\{g:\R^d\to\R \text{ measurable such that }\int_{\R^d}\l v\r^\ell g(v)\,dv<\infty\right\}.$$
One can easily see then that since $\gamma\geq 0$ and $b\in L^\infty([-1,1])$, the weak form \eqref{weak form} is applicable for any test function $\phi$ with $|\phi(v)|\le C \l v\r^i$ (since all integrations involved in the right hand side of \eqref{weak form} are justified).  Therefore,  using \eqref{weak form} and the conservation of  momentum and energy at the collisional level, we obtain
\begin{align}
&\int_{\R^d} Q(f,f)\phi(v)\,dv=0, \quad\forall t\in[0,T],\text{  a.e. $x\in \R^d$} \label{collisional averaging use}, 
\end{align}
where $\phi=1$ if $i=0$, $\phi\in\{1,v\}$ if $i=1$, and $\phi\in\{1,v,|v|^2\}$ if $i=2$.

Moreover, integrating  the first inequality of \eqref{pointwise bound} in space-velocity and taking supremum in time,  we obtain
\begin{align}
\sup_{t\in[0,T]}\int_{\R^{2d}}\l \beta v\r^{\gamma+i}f(t,x,v)\,v\,dx\,dv&\leq M\sup_{t\in[0,T]}\int_{\R^d}\left(\int_{\R^d}\l \alpha(x-vt)\r^{-p}\,dx\right)\l \beta v\r^{\gamma+i-q}\,dv\nonumber\\
& =M \alpha^{-d} \beta^{-d} \left(\int_{\R^d}\l x'\r^{-p}\,dx'\right)\left(
\int_{\R^d}\l v'\r^{\gamma+i-q}\,dv'\right)<\infty,\label{BE solution moments}
\end{align}
since $p>d$ and $q>d+\gamma+i$. Thus $f\in L^\infty([0,T],L^1_x L^{1,\gamma+i}_v)$.
Since $\gamma\ge 0$ and $b\in L^\infty([-1,1])$, one can easily see that $Q(f,f)\in L^\infty([0,T],L^1_x L^{1,i}_v)$. 

Now let $\phi=1$ if $i=0$, $\phi\in\{1,v\}$ if $i=1$, and $\phi\in\{1,v,|v|^2\}$ if $i=2$.
Integrating \eqref{definition of mild solution equation} in space-velocity and using Fubini's theorem and \eqref{collisional averaging use}, for any $t\in[0,T]$ we have
\begin{align}
\int_{\R^{2d}}f(t,x+tv,v)\phi(v)\,dx\,dv
&=\int_{\R^{2d}}f_0(x,v)\phi(v)\,dx\,dv+\int_0^t\int_{\R^d}\int_{\R^d} Q(f,f)(s,x+sv,v)\phi(v)\,dx\,dv\,ds\nonumber\\
&=\int_{\R^{2d}}f_0(x,v)\phi(v)\,dx\,dv+\int_0^t\int_{\R^d}\int_{\R^d} Q(f,f)(s,x,v)\phi(v)\,dx\,dv\,ds\nonumber\\
&=\int_{\R^{2d}}f_0(x,v)\phi(v)\,dx\,dv+\int_0^t\int_{\R^d}\int_{\R^d} Q(f,f)(s,x,v)\phi(v)\,dv\,dx\,ds\nonumber\\
&=\int_{\R^{2d}}f_0(x,v)\phi(v)\,dx\,dv\nonumber.
\end{align}
Thus for any $t\in[0,T]$, we have $$\int_{\R^{2d}}f(t,x,v)\,dx\,dv=\int_{\R^{2d}}f(t,x+vt,v)\,dx\,dv=\int_{\R^{2d}}f_0(x,v)\,dx\,dv,$$
and \eqref{conservation of mass equation}-\eqref{conservation of energy equation} follow.
\end{proof}

\subsection{Global well-posedness of the Boltzmann hierarchy for admissible data}
In this section, with the global well-posedness of the Boltzmann equation for small polynomially decaying initial data in hand,  we will construct a solution to the Boltzmann Hierarchy \eqref{BH} for initial data which is admissible, in the sense of Definition \ref{admissible data}, and for a range of values of  the chemical potential. To do this, we utilize a Hewitt-Savage representation \cite{hs55} tailored to our norms, in order to express any admissible datum as a convex combination of tensorized states under some appropriate probability measure.

More specifically, let us consider the set of probability densities  
\begin{equation} \label{P definition}
\mathcal{P}=\left\{h\in L^1(\R^{2d}): h\geq 0,\quad\int_{\R^{2d}}h(x,v)\,dx\,dv=1\right\}.
\end{equation}

\begin{proposition}[Hewitt-Savage] \label{hewitt savage}
    Suppose  $G=(g^{(k)})_{k=1}^\infty$ is admissible in the sense of Definition \ref{admissible data}. Then, there exists a unique Borel probability measure $\pi$ on $\mathcal{P}$
     such that 
    \begin{equation}\label{representation of marginal}
         g^{(k)} = \int_{\mathcal{P}} h^{\otimes k} d\pi(h), \qquad  \forall \, k \in \N.
    \end{equation}
    If additionally $G\in \mathcal{X}_{p,q,\alpha,\beta,\mu'}^\infty$,
 for some $p,q>1$, $\alpha,\beta>0$ and $\mu'\in\R$, then
    \begin{equation}\label{support inclusion}
       \support(\pi)\subseteq \left\{ h\in \mathcal{P} : \Vert h \Vert_{p,q,\alpha,\beta} \leq e^{-\mu'} \right\}.
    \end{equation}
\end{proposition}

\begin{proof}
 We note that similar versions of Hewitt-Savage can be found in the literature (see \cite[Proposition 6.1.3]{GSRT13}, \cite[Theorem 2.6]{FLS18}, \cite{RS23}, \cite{DH22}), but we will present a proof of the version used in this paper.
 
 Since $G \in \mathcal{A}$, we can view $G$ as being the law of a symmetric system of $\R^{2d}$ valued random variables as in the classical Hewitt-Savage theorem \cite{hs55}.  This furnishes a unique Borel probability measure $\pi$ over $\mathcal{P}$ so that representation \eqref{representation of marginal} holds.

 Now, assume  that $G\in \mathcal{X}_{p,q,\alpha,\beta,\mu'}^\infty$ in addition to being admissible.   Consider the set 
 \begin{equation}\label{support set} 
 E : = \{ h \in \mathcal{P} : \Vert h \Vert_{p,q,\alpha,\beta} \leq e^{-\mu'} \}.
 \end{equation} 
In order to establish \eqref{support inclusion} , we need to prove that $\pi(E^c) = 0$.
Let us define the function 
$$M(x,v) : = e^{-\mu '} \l \alpha x\r^{-p} \l \beta v\r^{-q}, $$ 
and a countable family of balls in $\mathbb{R}^{2d}$ as $\mathcal{B} = \cup_{n\in \mathbb{N}}\{B_{1/n}(x): x \in \mathbb{Q}^{2d} \}$, where we use the notation $B_r(y)$ to represent a ball in $\mathbb{R}^{2d}$ centered at $y$ with radius $r$. By the Lebesgue differentiation theorem, we can represent the set $E$ as: 
  \begin{equation} \label{E as intersection over balls}
     E = \bigcap_{B \in \mathcal{B}} \left\{ h \in \mathcal{P} : \int_{B} h(x,v) dx dv \leq \int_{B} M(x,v) dx dv \right\}.
 \end{equation}
Hence, taking complements, 
\begin{equation*}
    E^c = \bigcup_{B \in \mathcal{B}} \left\{ h \in \mathcal{P} : \int_{B} h(x,v) dx dv > \int_{B} M(x,v) dx dv \right\},
\end{equation*}
so by the countable sub-additivity of $\pi$, it suffices to show that 
\begin{equation} \label{support condition on a ball}
   \forall B \in \mathcal{B}, \qquad  \pi \left( \left\{ h \in \mathcal{P} : \int_{B} h(x,v) dx dv > \int_{B} M(x,v) dx dv \right\} \right) = 0.
\end{equation}

In order to prove \eqref{support condition on a ball}, fix $B \in \mathcal{B}$ and note that since $G$ is an element of $ X_{p,q, \alpha,\beta, \mu'}^\infty$ we have $\|G\|_{p,q,\alpha,\beta, \mu'} \ge e^{\mu' k }\l\l  \alpha X_k\r\r^p\l\l \beta V_k\r\r^q g^{(k)}(X_k, V_k)$ almost everywhere.  Hence 
\begin{equation}\label{G upper Mk}
\begin{split}
    \int_{B^k} g^{(k)}(X_k, V_k) dX_k dV_k 
    & \le \|G\|_{p,q,\alpha,\beta, \mu'}  \int_{B^k}  e^{-\mu' k }\l\l  \alpha X_k\r\r^{-p}\l\l \beta V_k\r\r^{-q} dX_k dV_k \\
    & =  \|G\|_{p,q,\alpha,\beta, \mu'} \left(\int_B M(x,v) dx dv \right)^k.
\end{split}
\end{equation}
Now, applying the representation \eqref{representation of marginal} to the left-hand side of \eqref{G upper Mk}, we have
\begin{align*}
   \int_{B^k} \int_{\mathcal{P}} h^{\otimes k} d\pi(h) dX_k dV_k \le \|G\|_{p,q, \alpha,\beta,\mu'} \left(\int_B M(x,v) dx dv \right)^k
\end{align*}
By applying Tonelli's theorem to the left-hand side, we have
\begin{align*}
    \int_{\mathcal{P}} \left( \int_B h(x,v) dx dv\right)^k d\pi(h) \le \|G\|_{p,q, \alpha,\beta, \mu'} \left(\int_B M(x,v) dx dv \right)^k.
\end{align*}
Therefore, using that $M > 0$ and $|B| > 0$ we have
\begin{align} \label{fundamental inequality for hs}
\int_{\mathcal{P}} \left(\frac{\int_B h(x,v) dx dv}{\int_B M(x,v) dx dv }\right)^k d\pi (h) \le  \|G\|_{p,q, \alpha,\beta,\mu'}.
\end{align}
Define 
\begin{equation*}
    \psi(h, B) : = \frac{\int_B h(x,v) dx dv}{\int_B M(x,v) dx dv },
\end{equation*}
and let $\epsilon > 0$. Then, by Chebyshev's inequality we have
\begin{equation} \label{chebyshev}
    \pi \left( \{ h \in \mathcal{P} : \psi(h,B) > 1+ \epsilon\} \right) \leq \frac{1}{(1 + \epsilon)^k} \int_{\mathcal{P}} \psi^k(h,B) d\pi (h). 
\end{equation}
This combined with \eqref{fundamental inequality for hs} implies 
\begin{equation}
    \pi \left( \{ h \in \mathcal{P} : \psi(h,B) > 1+ \epsilon\} \right) \leq \frac{\Vert G \Vert_{p,q, \alpha,\beta,\mu'}}{(1+ \epsilon)^k} \rightarrow 0
\end{equation}
as $k \rightarrow \infty$. Taking a countable sequence of $\epsilon \rightarrow 0$ implies \eqref{support condition on a ball}, finishing the proof.
\end{proof}

\begin{remark}\label{ball=space}
We note that representation \eqref{representation of marginal}, and the support condition \eqref{support inclusion}   imply that $\mathcal{A}\cap \mathcal{X}_{p,q,\alpha,\beta,\mu'}^\infty=\mathcal{A}\cap B_{\mathcal{X}_{p,q,\alpha,\beta,\mu'}^\infty}$, where $B_{\mathcal{X}_{p,q,\alpha,\beta,\mu'}^\infty}$ denotes the unit ball of $\mathcal{X}_{p,q,\alpha,\beta,\mu'}^\infty$.
\end{remark}

With Proposition \ref{hewitt savage} in hand, we can now construct global in time solutions to the Boltzmann hierarchy.

\begin{proof}[Proof of Theorem \ref{existence theorem BH}]
    Let $F_0=(f_0^{(k)})_{k=1}^\infty \in \mathcal{A}\cap \mathcal{X}_{p,q,\alpha,\beta,\mu'}^\infty$. Then by Proposition \ref{hewitt savage}   there exists  a Borel probability measure $\pi$ on $\mathcal{P}$ such that 
    \begin{equation}
     f_0^{(k)} = \int_{\mathcal{P}} h_0^{\otimes k} d\pi(h_0),
    \end{equation}
    and 
    \begin{equation}\label{support proof}
        \text{supp}(\pi)\subseteq \left\{ h_0 \in \mathcal{P} : \Vert h_0\Vert_{p,q,\alpha,\beta} \leq e^{-\mu'} \right\}.
    \end{equation}
Thus,   for $\pi$-almost any $h_0 \in \mathcal{P}$, we have that $\Vert h_0 \Vert_{p,q,\alpha,\beta}\leq e^{-\mu'}= \frac{e^{-\mu}}{2}$.
Now we may apply   Theorem \ref{BE is well posed} (for $M=e^{-\mu}$)  to construct a mild solution $h(t)$ of the Boltzmann equation with initial data $h_0$, which satisfies the bound
   \begin{equation}\label{bound on h existence proof}
   ||T_1^{-t}h(t)\|_{p,q, \alpha, \beta}\leq 2\|h_0\|_{p,q,\alpha,\beta} \le e^{-\mu},\quad\forall\, t\in[0,T].    
   \end{equation}   
Note that, given $t\in[0,T]$, the map $h_0\mapsto h(t)$ is continuous (and thus Borel measurable), due to continuity with respect to initial data estimate \eqref{stability estimate}.

   With this in hand,  we define $F=(f^{(k)})_{k=1}^\infty$, by
    \begin{equation}\label{solution representation: BH}
        f^{(k)}(t) : = \int_{\mathcal{P}} h(t)^{\otimes k} d\pi(h_0),\quad t\in[0,T],\,\, k\in\N.
    \end{equation}
Given $k\in\N$ and $t\in[0,T]$,  we have
\begin{align}
        e^{\mu k}\Vert T^{-t}_k f^{(k)}(t)\Vert_{k,p,q,\alpha,\beta} &\leq  e^{\mu k}\int_{\mathcal{P}} \Vert T^{-t}_k h^{\otimes k}(t)\Vert_{k,p,q,\alpha,\beta} d\pi(h_0)\nonumber \\
        & =  e^{\mu k}\int_{\mathcal{P}} \Vert T^{-t}_1 h(t)\Vert_{p,q,\alpha,\beta}^k d\pi(h_0) \label{use of tensorization of transport}\\
        &\leq 1\label{final estimate existence},
    \end{align}
where to obtain \eqref{use of tensorization of transport} we used \eqref{tensorization of transport norm}, and to obtain \eqref{final estimate existence} we used estimate \eqref{bound on h existence proof} and the fact that $\pi$ is a probability measure. Taking supremum in time estimate \eqref{sc 2 stability estimate hierarchy} follows. In particular, $\mathcal{T}^{-(\cdot)}F(\cdot)\in \mathcal{X}_{p,q,\alpha,\beta,\mu,T}^\infty$.

 Now, a standard computation shows that $F$ is a mild $\mu$-solution to the Boltzmann hierarchy \eqref{BH}, corresponding to initial data $F_0\in \mathcal{X}_{p,q,\alpha,\beta,\mu'}^\infty\subset \mathcal{X}_{p,q,\alpha,\beta,\mu}^\infty $, and this solution is unique due to Theorem \ref{Uniqueness theorem}. 
 
 The conservation laws
 \eqref{conservation of mass: BH}-\eqref{conservation of energy: BH}, follow from representation \eqref{solution representation: BH}, Fubini's theorem, and the conservation laws  \eqref{conservation of mass equation}-\eqref{conservation of energy equation} at the level of the Boltzmann equation. 

Finally, if additionally $F_0$ is tensorised, i.e. $F_0=(f_0^{\otimes k})_{k=1}^\infty$ with  $\|f_0\|_{p,q,\alpha,\beta}\leq e^{-\mu'}$, we prove  the stability estimate  \eqref{stability estimate hierarchy}. In that case $F=(f^{\otimes k})_{k=1}^\infty$, where $f$ is the mild solution of the Boltzmann equation with initial data $f_0$, obtained by Theorem \ref{BE is well posed}. In particular, by \eqref{bound wrt initial data equation}, we have $\|T_1^{-t}\|_{p,q,\alpha,\beta}\leq 2\|f_0\|_{p,q,\alpha,\beta}$. Therefore, using \eqref{tensorization of transport norm} and \eqref{norm of tensors}, we obtain
\begin{align*}
e^{\mu k}\|T_k^{-t}f^{\otimes k}(t)\|_{k,p,q,\alpha,\beta}&= e^{\mu k}\|T_1^{-t}f(t)\|_{p,q,\alpha,\beta}^k\leq 2^ke^{\mu k} \|f_0\|_{p,q,\alpha,\beta}^k\leq e^{\mu' k} \|f_0^{\otimes k}\|_{k,p,q,\alpha,\beta}\leq \|F_0\|_{p,q,\alpha,\beta,\mu',T}.
\end{align*}
Taking supremum over time, bound \eqref{stability estimate hierarchy} follows.

\end{proof}

\section{Appendix}\label{appendix}

Here we present the proofs of Lemma \ref{lemma on positions weight} and Lemma \ref{lemma on velocities weight}. As mentioned, for $d=3$ these proofs can be found in \cite{beto85} and \cite{to86} respectively. Inspired by these, we extend these results to arbitrary dimension $d\geq 3$.

\begin{proof}[Proof of Lemma \ref{lemma on positions weight}] Let us fix $t\geq 0$.  
Notice that for  $s\ge 0$ there holds
\begin{align}
&|x+s\xi|\geq |x| \Leftrightarrow 2sx\cdot\xi+s^2|\xi|^2\geq 0\Leftrightarrow s\geq \frac{-2 x\cdot \xi}{|\xi|^2},\label{iff ineq xi}\\
&|x+s\eta|\geq |x| \Leftrightarrow 2sx\cdot\eta+s^2|\eta|^2\geq 0\Leftrightarrow s\geq \frac{-2 x\cdot \eta}{|\eta|^2}.\label{iff ineq eta}
\end{align}
We define $h=\min\left\{ \frac{-2 x\cdot \xi}{|\eta|^2}, \frac{-2 x\cdot \eta}{|\eta|^2}\right\}$. Then we have the following cases:

Case 1: $0<h<t$. We can write
\begin{align*}\int_0^t \l  x+s\xi\r^{-p}\l x+s\eta\r^{-p}\,ds= I_1+I_2,
\end{align*}
where
$$I_1= \int_0^h \l x+s\xi\r^{-p}\l x+s\eta\r^{-p}\,ds,\quad I_2=\int_h^t \l x+s\xi\r^{-p}\l x+s\eta\r^{-p}\,ds.$$

We first estimate $I_1$. 
Fix $s\in[0,h]$, so $s\le \min\left\{ \frac{-2 x\cdot \xi}{|\eta|^2}, \frac{-2 x\cdot \eta}{|\eta|^2}\right\} $. Then \eqref{iff ineq xi}-\eqref{iff ineq eta} imply that $2sx\cdot\xi+s^2|\xi|^2$ and $2sx\cdot\eta+s^2|\eta|^2$ are non-positive, thus
\begin{equation}\label{condition on product}
   (2sx\cdot\xi+s^2|\xi|^2)(2sx\cdot\eta+s^2|\eta|^2)\geq 0.
\end{equation} 
Let us write $n:=\xi+\eta$. Since $\xi\cdot\eta=0$, we have $|n|^2=|\xi|^2+|\eta|^2>0$, since $\xi,\eta\neq 0$. Then, we obtain
\begin{align}
&\l x+s\xi\r^2\l x+s\eta\r^2\nonumber\\
&= (1+|x|^2+2sx\cdot\xi+s^2|\xi|^2)(1+|x|^2+2sx\cdot\eta+s^2|\eta|^2) \nonumber\\
&=(1+|x|^2)(1+|x|^2+2sx\cdot\eta+s^2|\eta|^2+2sx\cdot\xi+s^2|\xi|^2)+(2sx\cdot\xi+s^2|\xi|^2)(2sx\cdot\eta+s^2|\eta|^2)\nonumber\\
&\geq \left(1+|x|^2\right)\left(1+|x|^2+2sx\cdot(\xi+\eta)+ s^2(|\xi|^2+|\eta|^2)\right)\label{bound on traveling weights 1}\\
&\geq \l x\r^{2}(1+(s\left|n\right|+x\cdot \hat{n})^2)\nonumber,
\end{align}
where to obtain \eqref{bound on traveling weights 1} we used \eqref{condition on product}. Hence,  we have
\begin{align}
I_1 \leq \l  x\r^{-p}\int_0^h \left(1+(s\left|n\right|+x\cdot \hat{n})^2\right)^{-p/2}\,ds
    \leq\frac{\l  x\r^{-p}}{|n|}\int_{-\infty}^{\infty}(1+r^2)^{-p/2}\,dr
    \leq \frac{\sqrt{2}p}{p-1}\,\,\frac{\l  x\r^{-p}}{\min\{|\xi|,|\eta|\}},\label{estimate on I_1}
\end{align}
where we used the fact that $|n|^2=|\xi|^2+|\eta|^2\geq 2\min^2\{|\xi|,|\eta|\}$, as well as the integral bound $\int_{-\infty}^\infty (1+r^2)^{-p/2}\,dr\leq \frac{2p}{p-1}$.

Let us now estimate $I_2$. Consider $s\in[h,t]$, so either $s\geq -\frac{2x\cdot\xi}{|\xi|^2}$ or $s\geq -\frac{2x\cdot\eta}{|\eta|^2}$. Assume that $s\geq -\frac{2x\cdot\xi}{|\xi|^2}$.  Then, by \eqref{iff ineq xi} we have that $|x+s\xi|\geq |x|$. Therefore, using the triangle inequality, we obtain
\begin{align}
I_2&\leq \l x\r^{-p}\int_h^t (1+|x+s \eta|^2)^{-p/2}\,ds\nonumber \leq \l x\r^{-p}\int_{-\infty}^{+\infty} \left(1+\left(s|\eta|-|x|\right)^2\right)^{-p/2}\,ds\nonumber\\
&= \frac{\l  x\r^{-p}}{|\eta|}\int_{-\infty}^{+\infty} (1+r^2)^{-p/2}\,dr\nonumber\leq \frac{2p}{p-1}\,\, \frac{\l  x\r^{-p}}{|\eta|}. 
\end{align}
Now if $s\geq -\frac{2x\cdot\eta}{|\eta|^2}$, the same argument gives $I_2\leq \frac{2p}{p-1}\,\frac{\l  x\r^{-p}}{|\xi|}$. In either case, we have
\begin{equation}\label{estimate on I_2}
    I_2\le \frac{2p}{p-1}\,\,\frac{\l x\r^{-2p}}{\min\{|\xi|,|\eta|\}}.
\end{equation}
Combining \eqref{estimate on I_1}-\eqref{estimate on I_2}, we obtain \eqref{estimate on I}.

Case 2: $h>t$. We have $s<h$, for all $s\in[0,t]$, thus the same reasoning we used to compute $I_1$ above applies in that case  as well.

Case 3: $h<0$. We have $s>h$, for all $s\in[0,t]$, thus the same reasoning we used to compute $I_2$ applies in that case as well.

The proof is complete.

\end{proof}

Before we prove Lemma \ref{lemma on velocities weight}, we prove the following auxiliary convolution estimate

\begin{lemma}\label{convolution lemma}
Let $q>d+\gamma-1$. Then there exists a positive constant $L_q$ such that 
\begin{equation}\label{convolution estimate}\int_{\R^d}|y-v|^{\gamma-1}\l y\r^{-q}\,dy\le L_q,\quad\forall v\in\R^d. 
\end{equation}
\end{lemma}
\begin{proof}
We decompose as
$$\int_{\R^d}|y-v|^{\gamma-1}\l y\r^{-q}\,dy=\int_{|y-v|>\l y\r}|y-v|^{\gamma-1}\l y\r^{-q}\,dy+\int_{|y-v|<\l y\r}|y-v|^{\gamma-1}\l y\r^{-q}\,dy$$
We have
\begin{align}
\int_{|y-v|>\l y\r}|y-v|^{\gamma-1}\l y\r^{-q}\,dy&\leq \int_{\R^d}  \l y\r^{\gamma-1-q}\,dy\nonumber= \omega_{d-1}\int_0^\infty r^{d-1}(1+r^2)^{\frac{\gamma-1-q}{2}}\,dr\nonumber\\
&\leq \frac{\omega_{d-1}}{d}+\omega_{d-1}\int_1^{\infty}r^{d-2+\gamma-q}\,dr=\omega_{d-1}\left(\frac{1}{d}+\frac{1}{q-d-\gamma+1}\right),\label{convolution estimate 1}
\end{align}
since $\gamma\in (1-d,1]$ and $q>d-1+\gamma$.

We also have
\begin{align}
\int_{|y-v|<\l y\r}  |y-v|^{\gamma-1}\l y\r^{-q}\,dy&= \int_{|y-v|<1}  |y-v|^{\gamma-1}\l y\r^{-q}\,dy+\int_{1<|y-v|<\l y\r}|y-v|^{\gamma-1}\l y\r^{-q}\,dy\nonumber\\
&\leq \int_{|y-v|<1}|y-v|^{\gamma-1}\,dy+\int_{|y-v|>1}|y-v|^{\gamma-1-q}\,dy\nonumber\\
&=\omega_{d-1}\int_0^1 r^{d+\gamma-2}\,dr+\omega_{d-1}\int_1^\infty r^{d+\gamma-2-q}\,dr\nonumber\\
&=\omega_{d-1}\left(\frac{1}{d+\gamma-1}+\frac{1}{q+1-d-\gamma}\right),\label{convolution estimate 2}
\end{align}
since $\gamma\in (1-d,1]$ and $q>d-1+\gamma$. Combining \eqref{convolution estimate 1}-\eqref{convolution estimate 2}, estimate \eqref{convolution estimate} follows.

\end{proof}

\begin{lemma}\label{convolution lemma 2}
The following hold:
\begin{itemize}
    \item[(i)] Let $q>d$. Then, there exists a positive constant $\widetilde{L}_q$ such that
    \begin{equation}\label{q>d}
    \int_{\R^d}|y-v|^{1-d}\l y\r^{-q}\,dy\le \widetilde{L}_q\left(|v|^{1-d}+|v|\l v\r^{-q}\right),\quad \forall v\in\R^d.  
    \end{equation}
    \item[(ii)] Let $d-1<q\le d$. Then, there exists a positive constant $\widetilde{L}_q'$ such that
    \begin{equation}\label{q<d}
    \int_{\R^d}|y-v|^{1-d}\l y\r^{-q}\,dy\le \widetilde{L}_q' |v|^{1-q^*},\quad \forall v\in\R^d,    
    \end{equation}
    where we denote $q^*=\frac{q-d+1}{2}.$
\end{itemize}

\end{lemma}

\begin{proof}
 Assume first that $q>d$. Then we have 
\begin{align}
&\int_{\R^d}|y-v|^{1-d}(1+|y|^2)^{-q/2}dy\nonumber \\
&= \int_{|y-v|>\frac{|v|}{2}}|y-v|^{1-d}(1+|y|^2)^{-q/2}dy +\int_{|y-v|<\frac{|v|}{2}}|y-v|^{1-d}(1+|y|^2)^{-q/2}dy \nonumber  \\
&\le 2^{d-1}|v|^{1-d}\int_{\R^d} (1+|y|^2)^{-q/2}dy + \left(\frac{2}{3}\right)^q\l v\r^{-q}\int_{|y-v|<\frac{|v|}{2}}|y-v|^{1-d}\,dy\nonumber\\
&\le \frac{2^{d-1}\omega_{d-1}q}{d(q-d)}|v|^{1-d}+\left(\frac{2}{3}\right)^q\l v\r^{-q}\int_{|y|<\frac{3}{2}|v|}|y-v|^{1-d}\,dy\label{bound on y}\\
&\le \frac{2^{d-1}\omega_{d-1}q}{d(q-d)}|v|^{1-d}+\left(\frac{2}{3}\right)^q\l v\r^{-q}\int_{|y|<\frac{3}{2}|v|}|y-v|^{1-d}\,dy\nonumber\\
&\le \frac{2^{d-1}\omega_{d-1}q}{d(q-d)}|v|^{1-d}+ \left(\frac{2}{3}\right)^{q-1}\omega_{d-1}|v|\l v\r^{-q},\nonumber\\
&\le \widetilde{L}_q\left(|v|^{1-d}+|v|\l v\r^{-q}\right)
\end{align}
where to obtain \eqref{bound on y} we use the fact that when $|y-v|<|v|/2$, we have $|y|\leq |y-v|+|v|<\frac{3}{2}|v|$. 
Estimate \eqref{q>d} is proved.

Now assume that $d-1<q\le d$. Let us write $q^*=\frac{q-d+1}{2}.$ 
Integrating in spherical coordinates with axis in the direction of $v$, we have
\begin{align*}
 \int_{\R^d}|y-v|^{1-d}\l y\r^{-q}\,dy&=\int_{\R^d}(1+|v+z|^2)^{-q/2}|z|^{1-d}\,dz\\
 &=\omega_{d-2}\int_0^\pi\int_0^\infty (1+|v|^2+2|v|r\cos\theta+r^2)^{-q/2}\sin^{d-2}\theta\,dr\,d\theta\\
 &=\omega_{d-2}\int_0^\pi\int_{0}^\infty \left(1+|v|^2\sin^2\theta +\left(
r+|v|\cos\theta\right)^2\right)^{-q/2}\sin^{d-2}\theta\,dr\,d\theta\\
&\le\omega_{d-2}\int_0^\pi\int_{-\infty}^\infty (1+|v|^2\sin^2\theta+\rho^2)^{-q/2}\sin^{d-2}\theta\,d\rho\,d\theta\\
&\leq \omega_{d-2}\int_0^\pi\int_{-\infty}^\infty \frac{(\sin\theta)^{d-1-q^*}}{|v|^{q^*-1}}(1+\rho^2)^{\frac{-1+q^*-q}{2}}\,d\rho\,d\theta\\
&=\omega_{d-2}|v|^{1-q^*}\left(\int_0^\pi (\sin\theta)^{d-1-q^*}\,d\theta\right)\left(\int_{-\infty}^\infty (1+\rho^2)^{\frac{-1+q^*-q}{2}}\,d\rho\right)\\
&\leq \widetilde{L}_q' |v|^{1-q^*},
\end{align*}
since $d-1<q^*<q\le d$.
\end{proof}

Now we are in the position to prove Lemma \ref{lemma on velocities weight}.

\begin{proof}[Proof of Lemma \ref{lemma on velocities weight}]
Let us define 
$$I(v)=\int_{\R^d\times\mathbb{S}^{d-1}}\frac{|u|^{\gamma-1}}{\sqrt{1-(\hat{u}\cdot\sigma)^2}} \frac{\l v\r^{q}}{\l v^*\r^{q}\l v_1^*\r^{q}}\,d\sigma\,dv_1,
$$
where we denote $u=v_1-v$.
Notice that for any $\hat{n}\in \S^{d-1}$,  integration in spherical coordinates yields
\begin{align}\label{b integral}
\int_{\S^{d-1}} \frac{1}{\sqrt{1-(\hat{n}\cdot\sigma)^2}}\,d\sigma&\leq \omega_{d-2}\int_0^\pi \sin^{d-3}(\theta)\,d\theta \leq \omega_{d-2}\pi,  
\end{align}
where by $\omega_{d-2}$ we will denote the area of $\S^{d-2}$.

We assume first that $|v|\leq 1$. Fixing $v_1\in\R^d$ and $\sigma\in\S^{d-1}$, conservation of energy yields
$$\l v^*\r^2\l v_1^*\r^2=(1+|v^*|^2)(1+|v_1^*|^2)\geq 1+|v_1|^2,$$
so
$$|u|^{\gamma-1}\frac{\l v\r^{q}}{\l v^*\r^{q}\l v_1^*\r^{q}}\leq \frac{2^{q/2}|u|^{\gamma-1}}{(1+|v_1|^2)^{q/2}}.$$
Using Fubini's theorem, bound \eqref{b integral} and Lemma \ref{convolution lemma} we obtain
$$I(v)\leq 2^{q/2}\omega_{d-2}\pi   \int_{\R^{d}}\frac{|u|^{\gamma-1}}{(1+|v_1|^2)^{q/2}}\,dv_1\le 2^{q/2}\omega_{d-2}\pi L_q.    $$
Since $|v|\leq 1$ was arbitrary, we conclude that
\begin{equation}\label{U_q^1}
    \sup_{|v|\leq 1}I(v)\leq U_{q}^1.
\end{equation}

Now, fix $|v|> 1$. We decompose $\R^d$ as follows:
\begin{equation}\label{decomposition of R}
   \R^{d}=A\cup B:=\left\{v_1\in\R^d\,:\,|u|\leq\frac{|v|}{2}\right\}\cup \left\{v_1\in\R^d\,:\,|u|>\frac{|v|}{2}\right\}. 
\end{equation}
We first focus on the set $A$. 
For any $v_1\in A$ and $\sigma\in\S^{d-1}$, we have
$$2v^*\cdot v_1^*=|v^*|^2+|v_1^*|^2-|v_1^*-v^*|^2=|v|^2+|v_1|^2-|u|^2.$$
Moreover, by triangle inequality, we have 
$|v_1|=|v+u|\geq |v|-|u|\geq \frac{|v|}{2},$ thus
$$v^*\cdot v_1^*\geq \frac{|v|^2}{2}\Rightarrow |v^*|^2|v_1^*|^2\geq |v^*\cdot v_1^*|^2\geq \frac{|v|^4}{4}$$
This bound and conservation of energy yield $\l v^*\r^2\l v_1^*\r^2\geq \frac{1}{4}(1+|v|^2)^2$,
thus 
\begin{equation}\label{bound on ratio A}
\frac{\l v\r^{q}}{\l v^*\r^{q}\l v_1^*\r^{q}}\leq \frac{2^q}{(1+|v|^2)^{q/2}}.
\end{equation}
Write
$$I_A(v):=\int_{A\times\S^{d-1}}\frac{|u|^{\gamma-1}}{\sqrt{1-(\hat{u}\cdot\sigma)^2}} \frac{\l v\r^{q}}{\l v^*\r^{q}\l v_1^*\r^{q}}\,d\sigma\,dv_1
$$
Using \eqref{bound on ratio A}, Fubini's theorem and \eqref{b integral}, we obtain
\begin{align}
I_A(v)&\leq \frac{2^q\omega_{d-2}\pi }{(1+|v|^{2})^{q/2}}\int_{|u|\leq \frac{|v|}{2}}|u|^{\gamma-1} \,dv_1\nonumber\\
&= \frac{2^q\omega_{d-2}\omega_{d-1}\pi }{(1+|v|^2)^{q/2}}\int_0^{|v|/2}r^{d+\gamma-2}\,dr\nonumber\\
&\le\frac{4^q\omega_{d-2}\omega_{d-1}\pi }{(d+\gamma-1)2^{d+\gamma-1}}|v|^{d+\gamma-1-q}\nonumber\\
&\leq  \frac{4^q\omega_{d-2}\omega_{d-1}\pi}{(d+\gamma-1)2^{d+\gamma-1}}:=U_q^2\label{U_q^2}
\end{align}
since $\gamma\in(1-d,1]$, $q>d+\gamma-1$ and $|v|>1$.

Next, we tackle the integral 
\begin{equation}\label{B integral Toscani Lemma}
 I_B(v):=\int_{B\times\S^{d-1}} \frac{|u|^{\gamma-1}}{\sqrt{1-(\hat{u}\cdot\sigma)^2}} \frac{\l v\r^{q}}{\l v^*\r^{q}\l v_1^*\r^{q}}\,d\sigma\,dv_1 .
\end{equation}
Notice that  given $v_1\in B$ and $\sigma\in\S^{d-1}$, we have $|v_1^*| > \frac{|v|}{4}$ or $|v^*| > \frac{|v|}{4}$.
Indeed if $|v^*|,\,|v_1^*|\leq |v|/4$, \eqref{relative velocity magnitude} and the triangle inequality would imply $|u|=|v^*_1-v^*|\leq |v|/2$, which is not possible since $v_1\in B$.
Hence we can bound
\begin{equation}\label{decomposition of I_B}
    I_B(v)\le I_B^1(v)+I_B^2(v),
\end{equation}
where
\begin{align*}
I_B^1(v)&= \int_{\R^d\times\S^{d-1}} \frac{|u|^{\gamma-1}}{\sqrt{1-(\hat{u}\cdot\sigma)^2}} \frac{\l v\r^{q}}{\l v^*\r^{q}\l v_1^*\r^{q}}\mathds{1}_{[|v_1^*-v^*|>|v|/2,\,|v_1^*|>|v|/4]} \,d\sigma\,dv_1   \\
I_B^2(v)&= \int_{\R^d\times\S^{d-1}} \frac{|u|^{\gamma-1}}{\sqrt{1-(\hat{u}\cdot\sigma)^2}} \frac{\l v\r^{q}}{\l v^*\r^{q}\l v_1^*\r^{q}}\mathds{1}_{[|v_1^*-v^*|>|v|/2,\,|v_1^*|>|v|/4]} \,d\sigma\,dv_1 
\end{align*}
Without loss of generality, it suffices to estimate $I_B^1(v)$. Indeed, substituting $\sigma\mapsto -\sigma$, one can see that $I_B^2(v)=I_B^1(v)$. 

In order to estimate $I_B^1(v)$, we will use a Carleman-type representation \cite{Carleman}. In particular, we  use Proposition A.2 from \cite{grst11}, as well as \eqref{relative velocity magnitude} and \eqref{Carleman formula}, to express $I_B^1(v)$ as follows:
$$I_B^1(v)=2^{d-3}\int_{\R^d}\frac{\l v\r^q}{|v^*-v|^2\l v^*\r^q}\int_{E_{v,v^*}}\frac{|v_1^*-v^*|^{\gamma}}{|v_1^*-v||v_1^*-v^*|^{d-3}\l v_1^*\r^q}\mathds{1}_{[|v_1^*-v^*|>|v|/2,\,|v_1^*|>|v|/4]}\,d\pi(v_1^*)\,dv^*,$$
where given $v^*\in\R^{d}$, $E_{v,v^*}$ is the hyperplane given by
$$E_{v,v^*}=\left\{v_1^*\in\R^d: (v^*-v)\cdot(v_1^*-v)=0\right\},$$
and $d\pi$ is the induced surface measure on $E_{v,v^*}$. Notice that on $E_{v,v^*}$, we have 
\begin{equation}|v_1^*-v|^2+|v^*-v|^2=|v_1^*-v^*|^2,
\end{equation}
so we can bound $I_B^1(v)$ as follows:
\begin{align}
    I_B^1(v)&=2^{d-3}\int_{\R^d}\frac{\l v\r^q}{|v^*-v|^2\l v^*\r^q}\int_{E_{v,v^*}}\frac{|v_1^*-v^*|^{\gamma}\mathds{1}_{[|v_1^*-v^*|>|v|/2,\,|v_1^*|>|v|/4]}}{|v_1^*-v|\left(|v_1^*-v|^2+|v^*-v|^2\right)^{\frac{d-3}{2}}\l v_1^*\r^q}\,d\pi(v_1^*)\,dv^*\nonumber\\
     &\leq2^{d-3}\int_{\R^d} \frac{\l v\r^{q}}{|v^{*} - v|^{d-1}\l v^{*}\r^{q}} \int_{E_{v,v^*}}  \frac{|v^*-v_1^*|^{\gamma}}{|v_1^{*} - v| \l v_1^{*} \r^{q}}\mathds{1}_{[|v_1^*-v^*|>|v|/2,\,|v_1^*|>|v|/4]}  d \pi(v_1^*) dv^{*}.\label{I_1 est 2}
\end{align}
Let us define 
\begin{align*}
 L_B^1(v)&=      \int_{\R^d} \frac{\l v\r^{q}}{|v^{*} - v|^{d-1}\l v^{*}\r^{q}} \int_{E_{v,v^*}\cap[|v_1^*-v|\le |v^*-v|]}  \frac{|v^*-v_1^*|^\gamma}{|v_1^{*} - v| \l v_1^{*} \r^{q}}\mathds{1}_{[|v_1^*-v^*|>|v|/2,\,|v_1^*|>|v|/4]}  d \pi(v_1^*) dv^{*} \\
 \widetilde{L}_B^1(v)&=    \int_{\R^d} \frac{\l v\r^{q}}{|v^{*} - v|^{d-1}\l v^{*}\r^{q}} \int_{E_{v,v^*}\cap[|v^*-v|< |v_1^*-v|]}  \frac{|v^*-v_1^*|^\gamma}{|v_1^{*} - v| \l v_1^{*} \r^{q}}\mathds{1}_{[|v_1^*-v^*|>|v|/2,\,|v_1^*|>|v|/4]}  d \pi(v_1^*) dv^{*}
\end{align*}

By \eqref{I_1 est 2}, in order to estimate $I_B(v)$, it suffices to estimate $L_B^1(v)$ and $\widetilde{L}_B^1(v)$.

We first estimate the integral $L_B^1(v)$. Notice that given $v^*\in\R^d$, we can parametrize as follows: 
 $$E_{v,v^*}\cap[|v_1^*-v|\leq |v^*-v|]=\left\{v+(|v^*-v|\tan\theta)\,\hat{n},\quad\hat{n}\in\S^{d-1},\,\,\hat{n}\cdot(v^*-v)=0,\,\,\theta\in[0,\pi/4]\right\}.$$
 Thus writing $R=R(\theta)=|v^*-v|\tan\theta$, the elementary area is given by
$$\Delta \pi= \omega_{d-1}(R+\Delta R)^{d-1}-\omega_{d-1} R^{d-1}=\omega_{d-1}\Delta R \sum_{k=0}^{d-2} (R+\Delta R)^{d-2-k}R^k\simeq (d-1)\omega_{d-1}  R^{d-2}\Delta R,$$
which yields $$d\pi(v_1^*)=(d-1)\omega_{d-1} R^{d-2}(\theta)\,dR(\theta)=(d-1)\omega_{d-1} |v^*-v|^{d-1}\tan^{d-2}(\theta)\sec^2\theta\,d\theta.$$
Consequently
 \begin{align*} \frac{|v^*-v_1^*|^\gamma}{|v_1^*-v|}\,d\pi_{v_1^*}&=\frac{|v^*-v|^\gamma\sec^\gamma\theta}{|v^*-v|\tan\theta} 
 (d-1)\omega_{d-1}|v^*-v|^{d-1}\tan^{d-2}\theta\sec^2\theta \,d\theta\\
 & =(d-1)\omega_{d-1}|v^*-v|^{d+\gamma-2}\tan^{d-3}\theta\sec^{2+\gamma}\theta \,d\theta
 \end{align*}
Hence, we have
\begin{align}
  \int_{E_{v,v^*}\cap [|v_1^*-v|\leq |v^*-v|] }  &\frac{|v^*-v_1^*|^\gamma}{|v_1^{*} - v| \l v_1^{*} \r^{q}}  \mathds{1}_{[|v_1^*-v^*|>|v|/2,\,|v_1^*|>|v|/4]}d \pi(v_1^*)\nonumber\\
  &\leq \frac{4^q}{\l v\r^{q}}\int_{E_{v,v^*}\cap [|v_1^*-v|\leq |v^*-v|] }  \frac{|v^*-v_1^*|^\gamma}{|v_1^{*} - v| }  d \pi(v_1^*)\nonumber\\
  &=\frac{(d-1)\omega_{d-1} 4^q |v^*-v|^{d+\gamma-2}}{\l v\r^{q}}\int_0^{\pi/4} \tan^{d-3}\theta\sec^{2+\gamma}\theta\,d\theta\nonumber\\
  &\leq 4^q(d-1)\omega_{d-1} \max\{1,2^{\frac{2+\gamma}{2}}\}  |v^*-v|^{d+\gamma-2}\l v\r ^{-q}
\end{align}
Thus, since $q>d+\gamma-1$, Lemma \ref{convolution lemma} implies
\begin{align}
 L_B^1(v)&\leq  4^q(d-1)\omega_{d-1} \max\{1,2^{\frac{2+\gamma}{2}}\} \int_{\R^d} |v^*-v|^{\gamma-1}\l v^*\r^{-q}\,dv^*
 \nonumber\\
 &\leq 4^q(d-1)\omega_{d-1}\max\{1,2^{\frac{2+\gamma}{2}}\} L_q\label{bound on L}
\end{align}

Now we focus on  $\widetilde{L}_B^1$. Since $\gamma\leq 1$ and $|v^*-v_1^*|>|v|/2$ on the domain of integration,   we have
\begin{align}
  \widetilde{L}_B^1(v)&\le    \int_{\R^d} \frac{2^{1-\gamma}\l v\r^{q}|v|^{\gamma-1}}{|v^{*} - v|^{d-1}\l v^{*}\r^{q}} \int_{E_{v,v^*}\cap[|v^*-v|< |v_1^*-v|]}  \frac{|v^*-v_1^*|}{|v_1^{*} - v| \l v_1^{*} \r^{q}} \mathds{1}_{|v_1^*|>|v|/4]} d \pi(v_1^*) dv^{*} \label{first bound on L_B}
\end{align}
For the inner integral, we have
\begin{align}
 \int_{E_{v,v^*}\cap [|v^*-v|< |v_1^*-v|]} & \frac{|v^*-v_1^*|}{|v_1^{*} - v| \l v_1^{*} \r^{q}} \mathds{1}_{|v_1^*|>|v|/4]}\, d \pi(v_1^{* })\nonumber
 \\&= \int_{E_{v,v^*}\cap [|v^*-v|< |v_1^*-v|]}  \frac{\left(|v^*-v|^2+|v_1^*-v|^2\right)^{1/2}}{|v_1^{*} - v| \l v_1^{*} \r^{q}} \mathds{1}_{|v_1^*|>|v|/4]}\, d \pi(v_1^*)\nonumber \\
 &\leq \sqrt{2}\int_{E_{v,v^*}}  \frac{1}{ \l v_1^{*} \r^{q}} \mathds{1}_{|v_1^*|>|v|/4]} \,d \pi(v_1^*)\nonumber\\
 &= \sqrt{2}\int_{|v|/4}^\infty \frac{1}{(1+r^2)^{q/2}}\mathcal{H}_{d-2}(\S_r^{d-1}\cap E_{v,v^*})\,dr\label{use of co-area}\\
 &\leq \sqrt{2}\omega_{d-2}\int_{|v|/4}^\infty \frac{r^{d-2}}{(1+r^2)^{q/2}}\,dr\label{bound on co-area}\\
 &\leq \frac{4^{q+1-d}\sqrt{2}\omega_{d-2}|v|^{d-1-q}}{q+1-d},\label{final co-area,}
\end{align}
where to obtain \eqref{use of co-area} we use  the co-area formula, to obtain \eqref{bound on co-area} we use the fact that $\mathcal{H}_{d-2}(\S_r^{d-1}\cap E_{v,v^*})\leq \omega_{d-2}r^{d-2}$, and to obtain \eqref{final co-area,} we use the fact that $q>d-1$.

Combining \eqref{first bound on L_B}, \eqref{final co-area,}, we obtain 
\begin{equation}\label{final bound on L_B}\widetilde{L}_B^1(v)\leq \frac{4^{q+1-d}\sqrt{2}\omega_{d-2}}{q+1-d}|v|^{d-2+\gamma-q}\l v\r^q \int_{\R^d} |v^*-v|^{1-d}(1+|v^*|^2)^{-q/2}\,dv^*.
\end{equation}

If $q>d$, \eqref{q>d} from Lemma \ref{convolution lemma 2} yields
\begin{align}
\widetilde{L}_B^1(v)&\le \frac{4^{q+1-d}\sqrt{2}\omega_{d-2}}{q+1-d}|v|^{d-2+\gamma-q}\l v\r^{q}\widetilde{L}_q\left(|v|^{1-d}+|v|\l v\r^{-q}\right)\nonumber\\
&\leq \frac{4^{q+1-d}\sqrt{2}\omega_{d-2}}{q+1-d}\left(|v|^{\gamma-1-q}\l v\r^q+|v|^{d-1+\gamma-q}\right)\nonumber\\
&\leq  \frac{4^{q+1-d}\sqrt{2}\omega_{d-2}}{q+1-d}\left(2^{q/2}+1\right), \label{bound on L-tilde}
\end{align}
since $|v|>1$ and $\gamma\leq 1$.

If $\max\{d-1+\gamma,d-1\}<q\le d$, \eqref{q<d} from Lemma \ref{convolution lemma 2} yields
$$\widetilde{L}_B^1(v)\le  \frac{4^{q+1-d}\sqrt{2}\omega_{d-2}}{q+1-d}\widetilde{L}_q' |v|^{d-2+\gamma+q}\l v\r^q|v|^{1-q^*},$$
where $q^*=\frac{1}{2}\left(q-\max\{d+\gamma-1, d-1\}\right)$. Since $q^*>d-1+\gamma$ and $|v|>1$, we obtain 
\begin{align}
 \widetilde{L}_B^1(v) &\leq \frac{4^{q+1-d}\sqrt{2}\omega_{d-2}}{q+1-d}\widetilde{L}_q 2^{q/2}|v|^{d-1+\gamma -q^*} \le \frac{4^{q+1-d}\sqrt{2}\omega_{d-2}2^{q/2}}{q+1-d} \label{bound on L-tilde 2}.
\end{align}

By \eqref{bound on L}, \eqref{bound on L-tilde}, \eqref{bound on L-tilde 2}, we obtain $I_B^1(v)\leq U_q^3$. Since $I_B^2(v)=I_B^1(v)$, we have $I_B(v)\leq 2U_q^3$. By \eqref{U_q^2} and \eqref{decomposition of R}, we obtain $I(v)\leq U_q^2+2U_q^3$. Since $|v|>1$ was arbitrary, we have $$\sup_{|v|>1}I(v)\leq U_q^2+2U_q^3,$$ which combined with \eqref{U_q^1} yields \eqref{Uq def}.
\end{proof}

\end{document}